\documentclass[10pt,reqno]{amsart}
\pdfoutput=1
\usepackage{amsmath,amscd,amsfonts,amssymb,latexsym,mathrsfs,verbatim,url,color}
\usepackage{amsthm}
\usepackage{graphicx,epsf,epic,eepic,epsfig}
\usepackage{url}
\usepackage{tikz}
\usetikzlibrary{patterns}
\usepackage{hyperref}
\usepackage{hyperref}
\hypersetup{
    bookmarks=true,         
    unicode=false,          
    pdftoolbar=true,        
    pdfmenubar=true,        
    pdffitwindow=false,     
    pdfstartview={FitH},    
    pdftitle={My title},    
    pdfauthor={Author},     
    pdfsubject={Subject},   
    pdfcreator={Creator},   
    pdfproducer={Producer}, 
    pdfkeywords={keyword1} {key2} {key3}, 
    pdfnewwindow=true,      
    colorlinks=true,       
    linkcolor=blue,          
    citecolor=blue,        
    filecolor=blue,      
    urlcolor=blue      
    }

 \numberwithin{equation}{section}
\theoremstyle{plain}
   \newtheorem{theorem}{Theorem}[section]
   \newtheorem{proposition}[theorem]{Proposition}
   \newtheorem{lemma}[theorem]{Lemma}
   \newtheorem{corollary}[theorem]{Corollary}
   \newtheorem{problem}{Problem}
   \newtheorem{conjecture}[theorem]{Conjecture}
   
\theoremstyle{definition}
   \newtheorem{definition}{Definition}[section]
   \newtheorem{example}[theorem]{Example}

\theoremstyle{remark}
   \newtheorem{remark}[theorem]{Remark}

\author[P.~Br\"and\'en]{Petter Br\"and\'en}
\address{Department of Mathematics, Royal Institute of Technology, SE-100 44 Stockholm, Sweden}
\email{pbranden@kth.se}
\thanks{The first author is a Royal Swedish Academy of Sciences Research Fellow
  supported by a grant from the Knut and Alice Wallenberg
  Foundation. The research is also supported by the G\"oran Gustafsson Foundation.}
\author[M.~Chasse]{Matthew Chasse}
\address{Department of Mathematics, Royal Institute of Technology, SE-100 44 Stockholm, Sweden}
\email{chasse@kth.se}
       


\def\kkk{\kern.2ex\mbox{\raise.5ex\hbox{{\rule{.35em}{.12ex}}}}\kern.2ex}

\newcommand{\LP}{\mathscr{L{\kkk}P}}
\newcommand{\NN}{\mathbb{N}}

\newcommand{\PP}{\mathcal{P}}
\newcommand{\PPP}{\overline{\mathcal{P}}}

\renewcommand{\SS}{\mathcal{S}}
\newcommand{\SSS}{\overline{\mathcal{S}}}
\newcommand{\FF}{\mathscr{F}}

\newcommand{\RR}{\mathbb{R}}
\newcommand{\CC}{\mathbb{C}}

\newcommand{\KK}{\mathbb{K}}

\def\HD{\mathscr{D}_\mu}
\def\HDD{\overline{\mathscr{D}}_\mu}
\def\PPD{\mathcal{P}_\mu}

\renewcommand{\Im}{{\rm Im}}
\renewcommand{\Re}{{\rm Re}}

\def\newop#1{\expandafter\def\csname #1\endcsname{\mathop{\rm
#1}\nolimits}}

\newcommand{\beq}{\begin{equation}}  
\newcommand{\eeq}{\end{equation}}

\newop{per}
\newop{diag}
\newop{supp}
\newop{Proj}
\newop{JP}
\newop{rank}
\newop{Sym}
\newop{sign}
\newop{Int}
\newop{Cap}
\newop{PolO}
\newop{St}
\newop{glb}
\newop{MAP}
\newop{MOD}

\begin{document}

\title[Strip preservers]{Classification theorems for operators preserving zeros in a strip}

\begin{abstract}
We characterize all linear operators which preserve certain spaces of entire functions whose zeros lie in a closed strip.  Necessary and sufficient conditions are obtained for the related problem with real entire functions, and some classical theorems of de Bruijn and P\'olya are extended.  Specifically, we reveal new differential operators which map real entire functions whose zeros lie in a strip into real entire functions whose zeros lie in a narrower strip; this is one of the properties that characterize a ``strong universal factor'' as defined by de Bruijn.  Using elementary methods, we prove a theorem of de Bruijn and extend a theorem of de Bruijn and Ilieff which states a sufficient condition for a function to have a Fourier transform with only real zeros.    
\end{abstract}

\maketitle

\tableofcontents

\newpage

\section{Introduction and main results}
In 1950, a paper of N.~G.~de~Bruijn \cite{D} connected operators which contract the zeros of real entire functions towards the real axis and Fourier transforms which have only real zeros.  Heuristically, an eigenfunction of an operator which contracts zeros towards the real axis will have only real zeros, provided its zeros can be shown to lie in some strip.  Using this principle, de~Bruijn found sufficient conditions for a Fourier transform to have only real zeros, extending work by P\'olya \cite{P2}.  We address two connected questions: (1) Which linear operators preserve the set of entire functions whose zeros lie in a strip? and (2) Which linear operators contract the zeros of a real entire function closer to the real axis?

A fundamental notion in complex analysis is that of linear transformations which preserve spaces of functions whose zeros lie in a given subset of the complex plane.  Stability and zero localization techniques have been applied, for example, to establish log-concavity in combinatorics \cite{brenti}, total positivity in matrix theory \cite{schoenberg4}, and to describe phase transitions in statistical mechanics \cite{BB2,LS}.  Recently, the Kadison--Singer problem, which was open for more than fifty years, was resolved using zero localization methods \cite{MSS}.  It is well-known that the Riemann hypothesis is equivalent to the statement that the Riemann $\xi$-function, as originally defined \cite{P1}, has only real zeros.  Results concerning the zero loci of the Riemann $\xi$-function have therefore been pursued with sustained interest \cite{CC3,CSV,D,KK,KKL,LM,N,SGD}.  
A theme present in the work E. Laguerre, G. P\'olya \cite{P}, J. Schur \cite{PS}, and others, was stated in a succinct form by T. Craven and G. Csordas as an open problem \cite{CC3} (modified below for clarity).

\begin{problem}\label{gen-stability-problem}
Let $\SS\subset\CC[z]$ \textup{(}or $\SS\subset\RR[z]$\textup{)} be a set of polynomials which have zeros only in the set $\Omega\subset\CC$. 
Characterize all linear operators $T:\SS\to \SS\cup\{0\}$.
\end{problem}

For $\Omega=\RR$, P\'olya and Schur characterized the operators $T:\SS\to \SS\cup\{0\}$ of the form $T[x^k]=\gamma_k x^k$, $\gamma_k\in\RR$, $k=0,1,2,\dots$ \cite{PS}.  Craven and Csordas \cite{CCFW, CC3, CCGL} extended the theory of P\'olya--Schur, while posing a number of formidable open problems, including Problem \ref{gen-stability-problem}.   J. Borcea and the first author solved Problem \ref{gen-stability-problem} in the case that $\Omega$ is either a closed circular domain, a circle, or a line \cite{BB1}.  This result was subsequently extended to multivariate polynomials \cite{BB2}, and then to transcendental entire functions \cite{B2}.  For open circular domains, a characterization has recently been given by E. Melamud \cite{M}.  Problem \ref{gen-stability-problem} has remained open in the important cases that $\Omega$ is a (closed or open) strip, convex sector, or double sector.  Here we continue the classification program for the strip case, and extend work of de Bruijn \cite{D} on the zeros of trigonometric integrals. 
We completely solve Problem \ref{gen-stability-problem} for complex polynomials when $\Omega$ is a prescribed closed strip in the complex plane (Theorems \ref{strip} and \ref{trans-strip-char}).  This is arguably the most important case of Problem \ref{gen-stability-problem} which goes beyond circular domains.  A characterization for the open strip follows from the observations of Melamud.  Problem \ref{gen-stability-problem} is consequently solved when $\Omega$ is any image of the open or closed strip under a M\"obius transformation (Corollary \ref{cor-mobius-n}).

For questions pertaining to the Riemann hypothesis, solutions to Problem \ref{gen-stability-problem}, where the set $\Omega$ is a strip, are especially relevant.   
We investigate linear operators $T$ which preserve the set of real polynomials whose zeros lie in a strip, and give an equivalent condition that $T$ preserves a larger set of polynomials (Theorem \ref{real-strip-h-equiv}).  These results are applied to give sufficient conditions for differential operators to contract the zeros of real entire functions towards the real axis, as we describe below.

Suppose a linear operator $T$ is defined by 
\[
T\left(\int_{-\infty}^\infty F(t)e^{izt}dt\right)=\int_{-\infty}^\infty S(t)F(t)e^{izt}dt,
\]
where $S$ is a real entire function of order at most two. 
In \cite{D}, de~Bruijn applies the term \emph{strong universal factor} to a function $S$ if the associated operator $T$ has two properties when acting on the space of entire functions which possess Fourier transforms:

Let $f$ be an entire function which possesses a Fourier transform.
\renewcommand\theenumi{\Roman{enumi}} 
\begin{enumerate}\itemsep 2pt
\item \label{su1}
If the zeros of $f(z)$ lie in the strip $|\Im~z|\le\mu$, then the zeros of 
$T(f(z))$ lie in a strip $|\Im~z|<\mu'$, where $\mu'<\mu$ is independent of $f(z)$.
\item \label{su2}
If for any $\varepsilon>0$, all but a finite number of the zeros of $f(z)$ lie in the strip $\{z:|\Im~z|\le\varepsilon\}$, then the function $T(f(z))$ has only a finite number of non-real zeros.
\end{enumerate} 
\renewcommand\theenumi{\arabic{enumi}}
It was proved by de~Bruijn that the factor $S(t)=\exp(\lambda t^2)$, $\lambda>0$, corresponds to an operator which satisfies property \eqref{su1}, and more than fifty years later, Ki and Kim proved condition \eqref{su2} holds \cite{KK}.  We prove sharp sufficient conditions for large classes of differential operators which have property \eqref{su1} of strong universal factors (Theorems \ref{stab-strip-entire} and \ref{integral-shrink-entire}); it remains to be determined whether these operators also have property \eqref{su2}.  Afterwards, we return to the original motivation in de~Bruijn's paper \cite{D} and investigate kernels which have Fourier transforms with only real zeros.

We now state the main results along with the required notation.
Let $\Omega \subseteq \CC^m$. A polynomial $p(z_1,\ldots, z_m)\in \CC[z_1,\ldots, z_m]$ is \emph{$\Omega$-stable} if 
$$
z \in \Omega \mbox{ implies } p(z) \neq 0.
$$
\noindent
Let $\CC_n[z_1,z_2,\ldots,z_m]$ be the space of all polynomials in $\CC[z_1,z_2,\ldots,z_m]$ of degree at most $n$, let $\PP_n(\Omega)$ be the set of $\Omega$-stable polynomials in $\CC_n[z_1,z_2,\ldots,z_m]$, and let $\PP(\Omega) = \bigcup_{n=0}^\infty \PP_n(\Omega)$.  A linear operator $T$ is said to \emph{preserve $\Omega$-stability} if  $T(\PP(\Omega))\subseteq \PP(\Omega)\cup\{0\}$.  For notational convenience the symbols  $\PP_n(\Omega)$ and $\PP(\Omega)$ are used with the number of variables implicitly defined by the dimension of $\Omega$, and the space of univariate complex polynomials is denoted $\CC[z]$.
The next theorem reduces Problem \ref{gen-stability-problem} when $\mathcal{S}$ is a closed strip to the case when $\mathcal{S}$ is a half--plane. This yields algebraic and transcendental characterizations of strip preservers for polynomials (Theorem \ref{trans-strip-char}), and for transcendental entire functions (Theorem~\ref{entire-strip}).

\begin{theorem}\label{strip}
Let $H_1$ and $H_2$ be two parallel open half-planes in $\CC$ such that $\CC \setminus (H_1\cup H_2)$ is a closed strip with non-empty interior, and let $T : \CC_n[z] \rightarrow \CC[z]$. Then $T$ preserves $H_1 \cup H_2$-stability if and only if either
\renewcommand\theenumi{\roman{enumi}}
\begin{enumerate}
\item
$T$ has range of dimension at most one and is of the form
\[
T[f]=\alpha(f)p \text{ for } f\in\CC_n[z]
\]
where $\alpha:\CC_n[z]\to\CC$ is a linear functional and $p\in\PP(H_1\cup H_2)$; or

\item
$T$ satisfies \textup{(}a\textup{)} and \textup{(}b\textup{)} below. 
\begin{enumerate}
\item $T : \PP_n(H_1) \rightarrow \PP(H_1) \cup\{0\}$ or $T : \PP_n(H_2) \rightarrow \PP(H_1) \cup\{0\}$; 
\label{H1-stability}
\item $T : \PP_n(H_1) \rightarrow \PP(H_2) \cup\{0\}$ or $T : \PP_n(H_2) \rightarrow \PP(H_2) \cup\{0\}$.
\label{H2-stability}
\end{enumerate}
\end{enumerate}
\end{theorem}

\noindent
Note that by the characterization of half-plane stability preservers \cite[Theorem 4]{BB1}, Theorem \ref{strip} characterizes preservers of $H_1 \cup H_2$-stability in terms of zero-restrictions on the symbols of the operators (to be defined below).  When the strip in Theorem \ref{strip} has empty interior, and thus is a line, all $H_1 \cup H_2$-stability preservers are characterized in \cite[Theorem 3]{BB1}.  

An operator which preserves stability on a region $H_1\cup H_2$, as described in the statement of Theorem \ref{strip}, is said to be a \emph{strip preserver}.  More generally, an operator $T:\CC[z_1,z_2,\ldots,z_m]\to \CC[z_1,z_2,\ldots,z_m]$ is said to \emph{preserve} a set of polynomials $\mathcal{S}$ if $T(\mathcal{S})\subseteq \mathcal{S}\cup\{0\}$.  An $\Omega$-stable polynomial is simply called \emph{stable} in the case where $\Omega = H^n$, $n\in\NN$, and 
$$
H= \{z \in \CC : \Im(z)>0\}.
$$

For $\Omega\subset\CC^n$, let $\PPP(\Omega)$ be the closure of the set of polynomials $\PP(\Omega)$ under uniform limits on compact sets in $\CC^n$.
An operator $T$ which acts on $\CC[z]$ is extended to multivariate polynomials by declaring the other variables constant with respect to $T$, and also to formal power series by acting on each term of the series. The formal power series
$$ 
\overline{G}_T(z,w) = T(e^{-zw}) = \sum_{k=0}^\infty \frac{(-1)^kw^k}{k!}T(z^k) \in\CC[z][[w]] 
$$ 
\noindent
is said to be the \emph{symbol of T}.   When $T$ acts on a space of polynomials of bounded degree $n$, the \emph{algebraic symbol} $G_T=T((z+w)^n)$ is used.  In the proofs of Theorems \ref{P-suff-alg} and \ref{P-suff-tran}, a version of the symbol for operators acting on multivariate polynomial spaces appears, and the reader is referred to \cite[Section 1.1]{BB2} for its definition and applications.  
Let $\bar{\zeta}$ be the complex conjugate of $\zeta\in\CC$, and for a set $\Omega\subset\CC$, let the set of conjugates in $\Omega$ be denoted by
\[
\Omega^*=\{\bar{\zeta}\,:\,\zeta \in \Omega\}.
\]
Fix the notation 
\[
\SS_\mu\;:=\;\PP(\{z\,:\,|\Im~z|>\mu\}), 
\]
and $\SS_\mu(\RR):= \SS_\mu\cap\RR[z]$, for the sets of complex (resp. real) polynomials whose zeros lie in a closed strip of width $\mu$.  We state most results relative to the strip $|\Im~z|\le \mu$, since it has special relevance for real entire functions, and can be translated to any other strip by a linear transformation of the variable $z$.  

Note that the algebraic characterization for strip preservers on $\CC[z]$ is given by requiring the conditions in Theorem \ref{strip} to hold for all degrees $n$, and the criteria in \cite[Theorem 4]{BB1}.  The following transcendental characterization of strip preserving operators is also obtained. 

\begin{theorem}\label{trans-strip-char}
Let $T : \CC[z] \rightarrow \CC[z]$ be a linear operator and let $H_1=\{z\,:\,\Im~z>\mu\}$, where $\mu>0$.  Then $T$ preserves $S_\mu$ if and only if 
\renewcommand\theenumi{\roman{enumi}} 
\begin{enumerate}
\item
$T$ has range of dimension at most one and is of the form
\[
T[f]=\alpha(f)p \text{ for } f\in\CC_n[z]
\]
where $\alpha:\CC[z]\to\CC$ is a linear functional and $p\in\PP(H_1\cup H^*_1)$; or
\item
$T$ satisfies \textup{(}a\textup{)} and \textup{(}b\textup{)} below.
\begin{enumerate}\itemsep 2pt
\item\label{upper} $e^{i\mu w}\overline{G}_T(z,w)\in\PPP(H_1\times H)$~~or~~$e^{-i\mu w}\overline{G}_T(z,w)\in\PPP(H_1\times H^*)$; 
\item\label{lower} $e^{i\mu w}\overline{G}_T(z,w)\in\PPP(H^*_1\times H)$~~or~~$e^{-i\mu w}\overline{G}_T(z,w)\in\PPP(H^*_1\times H^*)$. 
\end{enumerate}
\end{enumerate}
\end{theorem}
\noindent 

 In the following examples, let $H_1$ be as in Theorem \ref{trans-strip-char} $T$ be a linear operator and let $T:\CC[z]\to\CC[z]$ unless otherwise stated.

\begin{example}\label{ex-finite-diff}
Let $T=\sum_{k=0}^N Q_k(z)D^k$ be a finite order differential operator with polynomial coefficients $\{Q_k(z)\}_{k=0}^N$, and let $F_T(z,w)=\sum_{k=0}^N Q_k(z)w^k$.  By Theorem \ref{trans-strip-char}, $T$ preserves $\SS_\mu$ if and only if
 \begin{align*}
e^{i\mu w}\overline{G}_T(z,w)=&\;F_T(z,w)e^{-(z-i\mu)w}\in\PPP(H_1\times H),  \text{~and~}\\
 e^{-i\mu w}\overline{G}_T(z,w)=&\;F_T(z,w)e^{-(z+i\mu)w}\in\PPP(H_1^*\times H^*).
\end{align*}
Since $e^{-(z-i\mu)w}\in\PPP(H_1\times H)$ and $e^{-(z+i\mu)w}\in\PPP(H_1^*\times H^*)$, we conclude that
\begin{center}
{\em $T$ preserves $\SS_\mu$ if and only if $F_T(z,w)\in\PP(H_1\times H)\cap\PP(H_1^*\times H^*)$.} 
\end{center}
Thus the criteria of Theorem \ref{trans-strip-char} simplify to just one criterion on the symbol.  
For an operator satisfying this criterion, $T : \PP_n(H_1) \rightarrow \PP(H_1) \cup\{0\}$  and $T : \PP_n(H_1^*) \rightarrow \PP(H_1^*) \cup\{0\}$. 
\end{example}

\begin{example}
The operator $T(p(z)) = p(-z)$ clearly preserves $\SS_\mu$ for any $\mu\ge 0$, and satisfies $T : \PP_n(H_1^*) \rightarrow \PP(H_1) \cup\{0\}$  and $T : \PP_n(H_1) \rightarrow \PP(H_1^*) \cup\{0\}$.
\end{example}

\begin{example}
The operator $T:\CC_n[z]\to\CC[z]$, given by $T(p(z)) = (z-1)^np(2iz/(z-1)),$  preserves $\SS_{\mu=1}$.  $T$ has been constructed using a M\"obius transformation which takes the open unit disk $\{z : |z|<1\}$ to the open half-plane $\{z : \Im~z<1\}$, and therefore $T : \PP_n(H_1) \rightarrow \PP(H_1) \cup\{0\}$  and $T : \PP_n(H_1) \rightarrow \PP(H_1^*) \cup\{0\}$.
\end{example} 

  The results in \cite{D} are dependent on the following theorem, which extends Jensen's theorem on critical points.  The notation $D:=d/dz$ is used throughout the sequel to simplify expressions involving derivatives.
\begin{theorem}[\cite{D}]\label{debruijn-jensen}
If $p(z)\in\SS_\mu(\RR)$, then for $\lambda \ge 0$, and $\xi\in\CC\setminus\{0\}$, 
$$\xi p(z+i\lambda) + \bar{\xi} p(z-i\lambda)\in\SS_\delta,$$  
where $\delta = \sqrt{\max\{\mu^2-\lambda^2, 0\}}.$
\end{theorem}
\noindent
(The original statement of Theorem \ref{debruijn-jensen} constrains the zeros of $\xi p(z+i\lambda) + \bar{\xi} p(z-i\lambda)$ to lie in a union of ellipses contained in the strip.) With $\xi=1/2$, the operator in Theorem \ref{debruijn-jensen} is $\cos(\lambda D)$, $\lambda\ge0$.  Using de Bruijn's theorem, D. Bleecker and  G. Csordas proved the following theorem, which describes precisely the action of $e^{-\lambda D^2}$ on real polynomials whose zeros lie in a horizontal strip.  

\begin{theorem}[{\cite[Theorem 3.2]{BC2001}}] \label{gaussian-diff}
If $p(z)\in\RR_n[z]\cap\SS_\mu$, $n\ge 1$, then for $\lambda\ge0$,
$$q(z) = e^{-\lambda D^2}p(z) = \sum_{k=0}^{\lfloor n/2 \rfloor} (-1)^k\frac{\lambda^k}{k!}D^{2k}p(z)\in\RR_n[z]\cap\SS_\delta,$$
where $\delta = \sqrt{\max\{\mu^2-2\lambda, 0\}}$.
\end{theorem}
\noindent
Note that Theorem \ref{gaussian-diff} follows quickly from Theorem \ref{debruijn-jensen} by observing the limit (uniform on compact sets)
\begin{equation}\label{cos-limit} 
\left(\cos\left(\sqrt{\frac{\lambda}{n}}D\right)\right)^n\to~ \exp(-\lambda D^2)\qquad (n\to\infty).
\end{equation}
H. Ki, Y. Kim, and J. Lee, proved generalizations of Theorem \ref{gaussian-diff} in the case where the polynomial $p$ is replaced by a transcendental entire function with a finite number of zeros outside of a prescribed strip \cite{KKL}.  
The differential operators furnished by Theorem \ref{debruijn-jensen} are limited to those of the form $A\cos(\lambda D + r)$, $a,\lambda,r\in\RR$. Prior to this work, the operators given by Theorem \ref{debruijn-jensen} together with $e^{-\lambda D^2}$, $\lambda>0$ were, to our knowledge, the only known operators that narrow the strip (up to composition).  The next theorem, proved  in Section \ref{real-case}, provides a large class of new operators which map $\SS_\mu(\RR)$ to $S_\delta(\RR)$, with $\delta<\mu$ for $\mu>0$.

\begin{theorem}\label{stab-strip}
Let $T=\sum_{k=0}^\infty a_k D^k = h(D)e^{i\lambda D}$ be a differential operator with constant coefficients $a_k\in\CC$, $k=0,1,2\ldots,$  $T^*:=\sum_{k=0}^\infty \bar{a}_k D^k$, and suppose that $h(w)\in\PPP(H^*)$.  If $p\in\SS_\mu(\RR)$, then 
$$(T+T^*)(p(z)) = h(D)p(z+i\lambda) + h^*(D)p(z-i\lambda) \in\SS_\delta(\RR),$$
where $\delta=\sqrt{\max\{\mu^2-\lambda^2, 0\}}$, and $h^*(z)=\overline{h(\bar{z})}.$
\end{theorem}

\begin{remark}
If $T=\sum_{k=0}^\infty a_k D^k$ is a differential operator as in Theorem \ref{stab-strip}, then $T:\PP(H)\to\PP(\{\Im~z>-\lambda\})$, $\lambda\ge0$, if and only if $T$ has the form $T=h(D)e^{i\lambda D}$, where $h(D)$ is a stability preserving differential operator. The corresponding symbol of such an operator is $\overline{G}_T(z,w)=h(-w)e^{-i\lambda w}$, where $h(-w)\in\PPP(H)$ (by \cite[Theorem 6]{BB1}) or equivalently $h(w)\in\PPP(H^*)$.  
\end{remark}

\begin{remark}
 Let $S(t)=h(it)e^{-\lambda t} + h^*(it)e^{\lambda t}$, where $h(w)\in\PPP(H^*)$. Theorem \ref{stab-strip} implies that functions with the form of $S$ satisfy property \eqref{su1} in the definition of strong universal factors (see Theorem \ref{stab-strip-entire}).  Recently, a preprint of D.~Cardon appeared  where he shows that an operator with property \eqref{su1} of strong universal factors must have order at least one and either mean or maximal type \cite{cardon}.
\end{remark}

Note that Theorem \ref{debruijn-jensen} is a special case of Theorem \ref{stab-strip}, where the operator $h$ is multiplication by $\xi$. 
In Section \ref{real-case}, Theorem \ref{stab-strip} is used to prove the following generalization of a theorem of P\'olya, who showed that the finite cosine transforms of positive increasing functions have only real zeros \cite[Problem 205]{PSz1}. 

\begin{theorem}\label{integral-shrink}
Let $f(z)$ be an entire function with the representation
$$
f(z) = \int_0^1 \cos(\lambda zt) g(t) \;dt, 
$$
where the function $g$ is non-negative and non-decreasing, and $\lambda\in\RR$.  The differential operator $f(D)$ maps  
$\SS_\mu(\RR)$ to $S_\delta(\RR)$, where $\delta = \sqrt{\max\{\mu^2-\lambda^2/4, 0\}}$.
\end{theorem}

Theorems \ref{stab-strip} and \ref{integral-shrink} provide a means of constructing a substantially more diverse collection of differential operators which map polynomials in $\SS_\mu(\RR)$ to polynomials whose zeros lie in a narrower strip.  For example, Theorem \ref{integral-shrink} implies that the zeroth Bessel function, $J_0(z)$ (see \eqref{bessel}), will shrink the strip containing the zeros of a real polynomial when applied as a differential operator $J_0(D)$.  We will frequently reference the following class of entire funtions.

\begin{definition}\label{lp}
A real entire function $\varphi(x)=\sum_{k=0}^{\infty}\gamma_kx^k/k!$ is in the 
\emph{Laguerre-P\'olya class}, written $\varphi\in \LP$, if it can be expressed in the form
$$\varphi(x) = cx^me^{-ax^2 + bx}\prod_{k=1}^\omega\left(1+\frac{x}{x_k}\right)e^{-x/x_k} \text{\hspace{5 mm}} (0\le\omega\le\infty),$$
where $b,c,x_k \in \mathbb{R}$, $m$ is a non-negative integer, $a\ge 0$, $x_k\neq 0$, and 
$\sum_{k=1}^\omega 1/x_k^2<\infty$. 
\label{LP}
\end{definition}
\noindent

Laguerre and P\'olya proved that $\LP=\PPP(H\cup H^*)$; that is, the Laguerre-P\'olya class consists precisely of functions which are locally uniform limits of polynomials which have only real zeros.
An immediate consequence of Theorem \ref{trans-strip-char} is that differential operators with constant coefficients, which preserve the set of polynomials whose zeros lie in a given strip, must be of the form $\varphi(D)$, where $\varphi$ is an entire function in the Laguerre-P\'olya class (Corollary \ref{strip-diff-op}).   

We continue with a series of lemmas and the proof of Theorems \ref{strip} and \ref{trans-strip-char} in Section \ref{strip-proof}.  M\"obius images of the strip are addressed in Appendix \ref{appendix1}.  In Section \ref{real-case}, necessary and sufficient conditions are given for a linear operator $T:\RR[z] \rightarrow \RR[z]$ to preserve $\SS_\mu$, $\mu \ge 0$ for real polynomials (Propositions \ref{real-strip-h-equiv}, \ref{real-strip}), although a characterization involving the operator symbol is not obtained.  These results are then applied to prove Theorems \ref{stab-strip} and \ref{integral-shrink}.  The main results are extended to transcendental entire functions in Bargmann-Fock spaces in Section \ref{section:b-f}  (see Theorems \ref{entire-strip}, \ref{stab-strip-entire} and \ref{integral-shrink-entire}).  In Section \ref{fourier-kernel}, we extend a theorem of N. G. de~Bruijn and L. Ilieff which states a sufficient condition for a Fourier transform to have only real zeros (Theorem \ref{fourier-polynomial}, Corollary \ref{fourier-cor}).  
\\[1ex]
\noindent
{\bf Ackowledgements.} Both authors would like to thank the American Institute of Mathematics for their hospitality during December of 2011 when this project was initiated.  We also thank Professor George Csordas for providing thorough feedback on an earlier version of the manuscript.

\section{Strip preservers for complex polynomials}
\label{strip-proof}

In this section, Theorem \ref{strip} is proved using a property of linear transformations which map complex polynomials stable on half-planes separated by a line, to polynomials stable on another region (Lemma \ref{line}).  At the end of the section, Theorem \ref{trans-strip-char} is proved, and a characterization for open strip preservers is given (Corollary \ref{open-strip}).  The classical result known as \emph{Hurwitz's Theorem} (see \cite[footnote 3, p. 22]{COSW} for a multivariate version) will be invoked frequently by name throughout the paper.  The proof of Theorem \ref{strip} requires several lemmas, including the following generalization of the classical Hermite-Biehler theorem.
 
\begin{lemma}[{\cite[Lemma 2.8]{BB2}}]\label{addz}
Let $p=q+ir \in \CC[z_1,\ldots, z_n]$, where $q$ and $r$ are real polynomials. The following are equivalent: 
\begin{enumerate}\renewcommand\theenumi{\roman{enumi}}
\item $p$ is stable; 
\item $q+wr \in \RR[z_1,\ldots, z_n,w]$ is stable. 
\end{enumerate}
\end{lemma}

\begin{remark}\label{trans-addz}
For later use we note that Lemma \ref{addz} extends readily to transcendental entire functions by Hurwitz's theorem.  The resulting statement is that $p=q+wr\in\PPP(H^n)$ if and only if $q+wr$ is real stable; that is, $q+wr\in\PPP(H^n\times H)\cap\RR[[z_1,\ldots,z_n,w]]$.
\end{remark}

\begin{lemma}\label{one}
Let $f(z_1,\ldots,z_n)$ and $g(z_1,\ldots,z_n)$ be multivariate entire functions, let $\Omega \subset \CC^n$ be a connected set, and let $C \subset \CC$ be a closed set such that $\CC \setminus C$ has exactly two connected components $A$ and $B$. 
If $h=f+z_{n+1} g$ is non-zero on $\Omega \times C$ and $g$ is non-zero on $\Omega$ (or $g \equiv 0$), then $h$ is either non-zero on $\Omega \times (A \cup C)$ or $\Omega \times (B \cup C)$. 
\end{lemma}

\begin{proof}
Clearly we may assume that $g \not \equiv 0$. By assumption 
$$
\zeta \in \Omega \mbox{ implies } - \frac{f(\zeta)}{g(\zeta)} \notin C.
$$
Connectivity prevents the event that $-f(\zeta_1)/g(\zeta_1)\in A$ and $-f(\zeta_2)/g(\zeta_2)\in B$ for some $\zeta_1, \zeta_2 \in \Omega$. Hence $-f(\zeta)/g(\zeta)\in A$ for all $\zeta \in \Omega$ or  $-f(\zeta)/g(\zeta)\in B$ for all $\zeta \in \Omega$, and the lemma follows.
\end{proof}

Lemma \ref{spaces} below describes linear spaces of multivariate entire functions which are non-zero on a given region.  A case of special relevance in Lemma \ref{spaces} is that of a linear space of stable polynomials, which is proved in \cite[Lemma 3.2]{BB2} in a less general form.

\begin{lemma}\label{spaces}
Let $V$ be a $\KK$-linear space of multivariate entire functions which are non-zero on $\Omega$, 
where $\KK=\RR$ or $\CC$ and $\Omega \subseteq \CC^n$ is a non-empty set.
\begin{enumerate}\renewcommand\theenumi{\roman{enumi}}
\item If $\KK=\RR$, $\Omega$ has non-empty interior and each non-zero element of $V$ is non-zero on $\Omega$,
then $\dim_\RR V \leq 2$.
\item
If $\KK=\CC$ and each non-zero element of $V$ is non-zero on $\Omega$,  
then $\dim_\CC V \leq 1$.
\end{enumerate}
\end{lemma}

\begin{proof}
We first prove (ii). Suppose that there are two linearly independent entire functions $f$ and $h$ in $V$. Then the function $G(\xi,\zeta)=f(\xi)+\zeta h(\xi)$ is non-zero on $\Omega$ for all $\zeta \in \CC$. Given any $\omega \in \Omega$, $h(\omega)\neq 0$ by definition, and there is a zero of $G(\omega,\zeta)$ at $\zeta_0 =- f(\omega)/h(\omega)$. Therefore, $f+\zeta_0 h$ has a zero in $\Omega$, and this is a contradiction. 

To prove (i) we may assume that $\Omega$ is open and connected, since otherwise we may replace $\Omega$ by an open and connected subset of $\Omega$. Suppose that there are three linearly 
independent functions $f_1,f_2$  and $f_3$ 
in $V$. Then $f_1+vf_2+wf_3$ is non-zero on $\Omega \times \RR^2$. By possibly multiplying $f_2$ or $f_3$ by $-1$, Lemma \ref{one} implies that 
$f_1+vf_2+wf_3$ is non-zero on $\Omega \times H^2$. The entire function $\lambda^{-1}( f_1+\lambda vf_2+\lambda wf_3)$ is non-zero on $\Omega \times H^2$ for all $\lambda>0$, so that by Hurwitz's Theorem $vf_2+wf_3$ is non-zero on $\Omega \times H^2$.  We claim that $f_2+\zeta f_3$ is non-zero on $\Omega$  for all $\zeta \in \CC$. Indeed, $w/v$ attains all values in $\CC$ except non-negative real values for $w,v\in H$. However, $f_2+\zeta f_3$ is assumed to be non-zero on $\Omega$ for real values of $\zeta$ also, so the claim follows.  This gives a contradiction by what has already been established in (ii). 
\end{proof}

\begin{lemma}\label{line}
Let $T : \CC_n[z] \rightarrow \CC[z]$ be a linear operator of rank greater than one.  Let $S\subset\CC$ be an open non-empty strip, let $L\subset S$ be any line in $S$, let the half-planes $H_1$ and $H_2$ be the connected components of the complement of $L$, and let $\Omega \subseteq \CC$ be a connected set with non-empty interior.  

If $T : \PP_n(S') \rightarrow \PP(\Omega)\cup\{0\}$, then $T : \PP_n(\overline{H_1}) \rightarrow \PP(\Omega)\cup\{0\}$ or $T : \PP_n(\overline{H_2}) \rightarrow \PP(\Omega)\cup\{0\}$. 
\end{lemma}
\begin{proof}
By a change of variables we may assume that $L$ is the real line and $H_1=H$, the open upper half-plane. Assume first that there is a polynomial $p(z)$ of degree $n$ with only real and simple zeros for which 
$T(p)(z) \equiv 0$. Then, by Hurwitz's theorem,  for each polynomial $q(z) \in \CC_n[z]$ there is a number $\epsilon>0$ for which $p(z)+\epsilon q(z)\in\PP(S')$. Since 
$\epsilon T(q) = T(p+\epsilon q)$, we have $T(q) \in \PP(\Omega)\cup\{0\}$. We conclude that $T(\CC_n[z]) \subseteq \PP(\Omega)\cup\{0\}$ and by Lemma \ref{spaces} that 
$\dim_\CC(T(\CC_n[z])) \leq 1$, which contradicts the hypothesis. Hence the image of each polynomial of degree $n$ with only real and simple zeros is $\Omega$-stable. 

Suppose that $p(z)=q(z)+ir(z)$ has all its zeros in $H_2$. By the Hermite--Biehler Theorem \cite[p.~197]{RS} and the Hermite-Kakeya Theorem \cite[p.~198]{RS}, $q(z)+\alpha r(z)$ has only real and simple zeros for all $\alpha \in \RR$. By what we observed above, $F_p(z,w) := T(q)(z)+wT(r)(z)$ is $\Omega \times \RR$-stable. Lemma \ref{one} now implies that $F_p(z,w)$ is either $\Omega \times \overline{H_1}$-stable or 
$\Omega \times \overline{H_2}$-stable. 

Now assume that there are two polynomials $p_j(z)=q_j(z)+ir_j(z)$, $j=1,2$, that have all their zeros in $H_2$ and such that $F_{p_j}(z,w)$ is $\Omega \times \overline{H_j}$-stable. Let 
$p_t$, $t \in [1,2]$, be a homotopy between $p_1$ and $p_2$ such that $F_{p_t}$ is $\Omega\times\RR$-stable for $1\le t\le 2$. It follows that there is a $t \in [1,2]$ for which $F(p_t)=T(q_t)+wT(r_t)$ is $\Omega \times \overline{H_1}$-stable and  
$\Omega \times \overline{H_2}$-stable. Thus, $F(p_t)$ is $\Omega \times \CC$ stable, which cannot happen as observed in the proof of Lemma \ref{spaces} (ii). 

We have proved that either $F_p$ is $\Omega \times H_1$-stable for all $p$ that have all zeros in $H_2$, or $F_p$ is $\Omega \times H_2$-stable for all $p$ that have all zeros in $H_2$. Setting $w=i$ in $F_p$, the Hermite--Biehler theorem now implies the conclusion of the lemma. 
\end{proof}

\begin{remark}\label{line-lemma-remark}
If $S$ is an open strip or a closed strip with non-zero interior and  $T : \PP_n(S') \rightarrow \PP(\Omega)\cup\{0\}$ is a linear operator with rank greater than one, then for any line in $S$ with complement $H_1\cup H_2$, 
\begin{equation}\label{eq-flip}
T : \PP_n(H_1) \rightarrow \PP(\Omega)\cup\{0\} \text{ or } T : \PP_n(H_2) \rightarrow \PP(\Omega)\cup\{0\}.  
\end{equation}
This follows immediately from Lemma \ref{line} and a limiting argument, by first showing \eqref{eq-flip} for the open strip, and then letting the lines approach the edge from the interior of a closed strip.
\end{remark}

We now give the proofs of Theorems \ref{strip} and \ref{trans-strip-char}, which are stated in the introduction.  

\begin{proof}[Proof of Theorem \ref{strip}]
If $T$ has rank one or less then it is clear that it must have form (i), whose sufficiency is clear also, so assume otherwise.
If $T$ satisfies conditions (a) and (b) of (ii), then it follows immediately that $T$ satisfies one of four conditions:
\begin{enumerate}
\item $T:\PP(H_1)\to\PP(H_1)\cup\{0\}$ and $T:\PP(H_1)\to\PP(H_2)\cup\{0\}$,
\item $T:\PP(H_1)\to\PP(H_1)\cup\{0\}$ and $T:\PP(H_2)\to\PP(H_2)\cup\{0\}$,
\item $T:\PP(H_1)\to\PP(H_2)\cup\{0\}$ and $T:\PP(H_2)\to\PP(H_1)\cup\{0\}$,
\item $T:\PP(H_2)\to\PP(H_1)\cup\{0\}$ and $T:\PP(H_2)\to\PP(H_2)\cup\{0\}$.
\end{enumerate} 
In any case, $T:\PP(H_1\cup H_2)\to\PP(H_1\cup H_2)\cup\{0\}$, and thus the sufficiency direction is clear.

To prove necessity, suppose that $T$ preserves $H_1 \cup H_2$-stability. For an open set $U,$ define the superscript notation
$$
U^r= \Int(\CC \setminus U).
$$
By Lemma \ref{line} and Remark \ref{line-lemma-remark}, either $T: \PP_n(H_1) \rightarrow \PP(H_1)\cup\{0\}$, or $T: \PP_n(H_1^r) \rightarrow \PP(H_1)\cup\{0\}$.
We first prove the following claim, from which the proof of necessity quickly follows. 
\\
\\
{\bf Claim : } \emph{If $T: \PP_n(H_1^r) \rightarrow \PP(H_1)\cup\{0\}$, then $T: \PP_n(H_2) \rightarrow \PP(H_1)\cup\{0\}$ or $T: \PP_n(H_1) \rightarrow \PP(H_1)\cup\{0\}$ .} 
\\
\\
Let $0 \leq t \leq 1$, and define
$$
K_t = \{ z\in\CC :  z=t z_1 + (1-t) z_2,~z_1\in H_2,~z_2\in H_1^r\}  
$$
one may check that $K_t$ is a half-plane for each $t\in[0,1]$, and that it is a continuous parametrization between $K_0=H_1^r$ and $K_1=H_2$.  
As above, Lemma \ref{line} and Remark \ref{line-lemma-remark} imply that $T: \PP_n(K_t) \rightarrow \PP(H_1)\cup\{0\}$ or $T: \PP_n(K_t^r) \rightarrow \PP(H_1)\cup\{0\}$ for each $0 \leq t \leq 1$. If $T: \PP_n(K_t) \rightarrow \PP(H_1)\cup\{0\}$ for all $0 < t <1$, then by Hurwitz's theorem  $T: \PP_n(K_1) \rightarrow \PP(H_1)\cup\{0\}$ and we have proved the claim. Assume therefore that $T: \PP_n(K_s^r) \rightarrow \PP(H_1)\cup\{0\}$ for some $0<s<1$. We separately address the following two possible cases:  
\begin{itemize}
\item[(A)] There is an $\epsilon>0$ for which $T: \PP_n(K_t) \rightarrow \PP(H_1)\cup\{0\}$ for all $0 \leq t <\epsilon$. 
\item[(B)] There is no $\epsilon>0$ for which $T: \PP_n(K_t) \rightarrow \PP(H_1)\cup\{0\}$ for all $0 \leq t <\epsilon$. 
\end{itemize}
If (A) holds, then by Hurwitz's theorem  there is a smallest $s$, $0<s<1$, such that $T: \PP_n(K_s^r) \rightarrow \PP(H_1)\cup\{0\}$.  To show this, let $S$ be the set of all values of $t$ for which $T: \PP_n(K_t^r) \rightarrow \PP(H_1)\cup\{0\}$, and suppose that $s_{\text{min}}=\inf S\not\in S$.  Then there must exist a countable sequence of values $\{s_\nu\}_{\nu=0}^\infty$ in $S$ with $s_\nu\to s_{\text{min}}$.  Since $T: \PP_n(K_{s_\nu}^r) \rightarrow \PP(H_1)\cup\{0\}$ for all $\nu=0, 1, 2, \ldots$, and $H_1$ is an open set, it follows by Hurwitz's Theorem that $T: \PP_n(K_{s_{\text{min}}}^r) \rightarrow \PP(H_1)\cup\{0\}$, which contradicts our assumption $s_{\min}\not\in S$. Therefore, it must be that $s_{\text{min}}\in S$, and henceforth let $s=s_{\min}$.  Note that by Hurwitz's theorem, $T: \PP_n(K_s) \rightarrow \PP(H_1)\cup\{0\}$ as well, by approaching $s$ from below.  Thus, $T: \PP_n(K_s^r)\cup\PP_n(K_s) \rightarrow \PP(H_1)\cup\{0\}$.  Let $G(z,w)= T[(z-w)^n]$. By what we have observed, for each $v \in \CC$, $G(z,v) \in \PP(H_1)\cup\{0\}$, since $(z-v)^n\in \PP_n(K_s)\cup \PP_n(K_s^r)$ for all $v\in\CC$. Suppose that $G(z,v) \equiv 0$ for some $v \in \CC$. If $v$ is in the interior of the strip, then for each $q(z) \in \CC_n[z]$ there is an $\epsilon >0$ such that all zeros of $(z-v)^n +\varepsilon q(z)$ are in the interior of the strip. It follows that $T(q)=T(\varepsilon q + G(z,v))/\varepsilon \in \PP(H_1)\cup\{0\}$, thus $T(\CC_n[z])\subseteq\PP(H_1)$, and we arrive at a contradiction by Lemma \ref{spaces}. When $v$ is not in the interior of the strip, it is in the interior of either $K_s$ or $K_s^r$.  Since $T: \PP_n(K_s^r) \rightarrow \PP(H_1)\cup\{0\}$ and 
$T: \PP_n(K_s) \rightarrow \PP(H_1)\cup\{0\}$, the same argument may be applied to show that if  $G(z,v) \equiv 0$ for $v\in\CC$ outside the interior of the strip, then $T(\CC_n[z])\subseteq\PP(H_1)$, and we arrive at a contradiction by Lemma \ref{spaces} again.  We conclude that $G(z,w)$ is $H_1 \times \CC$-stable. By the fundamental theorem of algebra, this cannot happen unless $T(z^k) \equiv 0$ for all $k >1$; i.e., $T$ has rank at most one.  Since by assumption the rank of $T$ is $2$ or more, case (A) cannot occur. 

If (B) holds, then by Hurwitz's theorem  $T: \PP_n(H_1) \rightarrow \PP(H_1)\cup\{0\}$, and this completes the proof of the claim.

From Lemma \ref{line} and the claim, we may conclude that $T: \PP_n(H_1) \rightarrow \PP(H_1)\cup\{0\}$ or $T: \PP_n(H_2) \rightarrow \PP(H_1)\cup\{0\}$.  By symmetry, the same statement must hold when $\PP(H_2)$ is the range; that is, $T: \PP_n(H_1) \rightarrow \PP(H_2)\cup\{0\}$ or $T: \PP_n(H_2) \rightarrow \PP(H_2)\cup\{0\}$.  This completes the proof of necessity.
\end{proof}

\begin{proof}[Proof of Theorem \ref{trans-strip-char}] Set $H_2=H^*_1$, and let $L_1(p(z))=p(z-i\mu)$ and $L_2(p(z))=p(-z-i\mu)$, be linear changes of variables which map $\PP(H)$ to $\PP(H_1)$ and $\PP(H_2)$ respectively.
By the transcendental characterization of stability preservers \cite[Theorem 6]{BB1} and Theorem \ref{strip},  satisfying both \eqref{upper} and \eqref{lower}  in Theorem \ref{trans-strip-char} (also see conditions (I) and (II) below) implies that $T$ preserves $H_1 \cup H_2$-stability. For the opposite direction, if $T$ preserves $H_1 \cup H_2$-stability, then Theorem \ref{strip} part (ii) holds for all degrees $n$.  From the algebraic characterization of stability preservers  in \cite[Theorem 4]{BB1}, conditions (a) and (b) of Theorem \ref{strip}, for an operator of rank one or more, are equivalent to the following conditions (with the identically zero function forbidden):
\renewcommand\theenumi{\Alph{enumi}}
\begin{enumerate}
\item\label{alg1} $L^{-1}_1TL_1((z+w)^n)\in\PP(H\times H)\cup\{0\}$ or $L^{-1}_1TL_2((z+w)^n)\in\PP(H\times H)\cup\{0\}$; 
\item\label{alg2} $L^{-1}_2TL_1((z+w)^n)\in\PP(H\times H)\cup\{0\}$ or $L^{-1}_2TL_2((z+w)^n)\in\PP(H\times H)\cup\{0\}$. 
\end{enumerate}
\noindent
If one of the conditions in \eqref{alg1} or \eqref{alg2} holds, then it holds for all degrees $m\le n$ as well \cite[Lemma 4]{BB1}.  Thus if $T$ preserves $H_1\cup H_2$-stability on $\CC[z]$, then one of the conditions in \eqref{alg1} holds for all $n$ and one of the conditions in \eqref{alg2} holds for all $n$.   The statement of the theorem then follows from the equivalence of the algebraic and transcendental characterizations of stability preservers given in \cite[Corollary 2]{BB1} and \cite[Theorem 6]{BB1}.  Namely, an operator $T$ of rank one or more preserves $H_1\cup H_2$ stability if and only if it satisfies 
\renewcommand\theenumi{\Roman{enumi}}
\begin{enumerate}
\item $L^{-1}_1TL_1(e^{-zw})\in\PP(H\times H)\cup\{0\}$ or $L^{-1}_1TL_2(e^{-zw})\in\PP(H\times H)\cup\{0\}$; 
\item $L^{-1}_2TL_1(e^{-zw})\in\PP(H\times H)\cup\{0\}$ or $L^{-1}_2TL_2(e^{-zw})\in\PP(H\times H)\cup\{0\}$. 
\end{enumerate}
\noindent
These conditions simplify to those stated in the theorem.
\end{proof}

An operator $T$ which preserves $\SS_\mu(\RR)$, $\mu\ge0$, will be referred to as a \emph{real strip preserver} (to be used in the sequel).
If the operator $T$ in Theorem \ref{trans-strip-char} commutes with the translation operator and the strip $\CC\setminus(H_1\cup H_2)$ is symmetric about the real axis, then the conditions in Theorem \ref{trans-strip-char} simplify to those for preservers of polynomials with only real zeros, which are characterized in \cite[Theorem 5]{BB1}.  This provides a characterization of the differential operators with constant coefficients which preserve the set of polynomials whose zeros all lie in a strip.  

\begin{corollary}\label{strip-diff-op}
If $f(z)$ is a formal power series with real coefficients, then $f(D)$ preserves $\SS_\mu$, $\mu\ge0$ if and only if $f\in\LP$.
\end{corollary}
\noindent
The sufficiency direction of Corollary \ref{strip-diff-op} was already known to P\'olya \cite[pp 155--158]{RS}.

A \emph{multiplier sequence} is a real sequence, $\{\gamma_k\}_{k=0}^\infty$, which defines a linear operator $T:\RR[z]\to\RR[z]$ by $T(z^k)=\gamma_kz^k$ that preserves reality of zeros.  It is elementary to show that the terms of a multiplier sequence must either have the same sign or alternate in sign.  G. Csordas and T. Craven characterized multiplier sequences that map any polynomial $p\in\CC[z]$ to a polynomial whose zeros lie in the convex hull formed of the zeros of $p$ and the origin \cite{CCGL}: the sequence $\{\gamma_k\}_{k=0}^\infty$ must be a constant multiple of a non-negative, non-decreasing multiplier sequence.  In addition to these sequences, those with alternating signs also correspond to strip preserving operators.

\begin{proposition}\label{c-strip-mult}
Let $\{\gamma_k\}_{k=0}^\infty$ be a real sequence.  The linear operator $T:\CC[z]\to\CC[z]$ defined by $T(z^k)=\gamma_kz^k$ preserves $\SS_\mu$ if and only if either $\{\gamma_k\}_{k=0}^\infty$ or $\{(-1)^k\gamma_k\}_{k=0}^\infty$ is a constant multiple of a non-decreasing, non-negative multiplier sequence.
\end{proposition}

\begin{proof}
Let the exponential generating function for the sequence $\{\gamma_k\}_{k=0}^\infty$ be denoted by
\[
\varphi(z) = \sum_{k=0}^\infty \frac{\gamma_k}{k!}z^k.
\]
The symbol of $T$ is then $\overline{G}_T(z,w)=\varphi(-zw)$.  By Theorem \ref{trans-strip-char}, $T$ is a strip preserver if and only if
\[e^{i\mu w}\varphi(-zw)\in\PPP(H_1\times H) \text{~or~} e^{i\mu w}\varphi(zw)\in\PPP(H_1\times H)\]
 and  
\[e^{i\mu w}\varphi(-zw)\in\PPP(H_2\times H) \text{~or~}  e^{i\mu w}\varphi(zw)\in\PPP(H_2\times H).\]
In any case, it is necessary that for all $w=r\in\RR$, 
\[
e^{ir\mu}\varphi(-rz)\in\PPP(H_1)\cap\PPP(H_2), 
\]
and therefore $\varphi(z)$ has only real zeros.  Suppose that $e^{i\mu w}\varphi(-zw)\in\PPP(H_1\times H)$.  Then setting $w=ir$, $r>0$, implies that $e^{\mu w}\varphi(-irz)\in\PPP(H_1\times H)$, which restricts $\varphi(z)$ to have only real non-positive zeros, and furthermore to be a function which is a locally uniform limit of polynomials with only real non-positive zeros.  One may conclude that the order of $\varphi$ is at most one, and that $\{\gamma_k\}_{k=0}^\infty$ is a multiplier sequence \cite[Chapter VII]{Levin}. In the other case, when $e^{i\mu w}\varphi(zw)\in\PP(H_1\times H)$, $\varphi(z)$ must have only real non-negative zeros ($\varphi(-z)$ has only real non-positive zeros).  

Assume that $\varphi$ has only real non-positive zeros.  According to the observations above, $\varphi$ is locally uniform limit of polynomials with only real non-positive zeros, and it therefore possesses the Hadamard factorization \cite[Chapter VIII]{Levin} 
\[
\varphi(z) = ce^{az}z^m\prod_{k=1}^\nu \left(1+\frac{z}{z_k}\right), \qquad (0\le\nu\le\infty),
\]
where $a\ge 0$, $c\in\RR$, $m$ is a non-negative integer, $z_k>0$ for $k=1,2,3,\ldots$, and $\sum_{k=1}^\infty 1/z_k <\infty$.
Then, 
\[
e^{i\mu w}\overline{G}_T(z,w) = ce^{-aw(z-i\mu/a)}(-zw)^m\prod_{k=1}^\nu \left(1-\frac{zw}{z_k}\right),
\]   
is in $\PP(H_1\times H)$ if and only if $a\ge 1$.  From \cite[Lemma 2.2]{CCGL} of Craven and Csordas, it follows that for $c>0$, $0\le\gamma_0\le\gamma_1\le\gamma_2\le\cdots.$  In the case that $\varphi$ has only real non-negative zeros, the same holds for $\{(-1)^k\gamma_k\}_{k=0}^\infty$. 
\end{proof}

With the following theorem of E. Melamud, a characterization for linear operators preserving the set of polynomials whose zeros lie in an open strip is obtained. 

\begin{theorem}[\cite{M}]\label{open-set}
Let $T:\CC[z]\to\CC[z]$ be a linear operator, $T:\PP(\Omega)\to\PP(\Omega)\cup\{0\}$, and $\Omega\subset\CC$ be an open set. Then $T:\PP(\overline{\Omega})\to\PP(\overline{\Omega})$ if and only if $T[z^k]\in\PP(\overline{\Omega})$, where $k = \min\{ m :T[z^m]\not\equiv 0\}.$
\end{theorem}
\noindent

\begin{corollary}\label{open-strip}
Let $T:\CC[z]\to\CC[z]$ be a linear operator. $T$ preserves the set of polynomials whose zeros lie only in the open strip $|\Im~z| <\mu$, $\PP(\{|\Im~z|\ge \mu\})$, if and only if it preserves $\SS_\mu$ and $T[z^k]\in\PP(\{|\Im~z|\ge \mu\})$, where $k = \min\{ m :T[z^m]\not\equiv 0\}.$  
\end{corollary}\noindent
The latter requirement in Corollary \ref{open-strip} may be stated in different forms.  In particular, $z^k$ may be replaced with any polynomial $p\in\CC[z]$ with $\deg(p)=k.$  Characterizations of M\"obius images of the strip are provided in Appendix \ref{appendix1}.

\section{Strip preservers for real polynomials}
\label{real-case}

We now seek a characterization of operators which preserve the set of real polynomials whose zeros lie in a prescribed closed strip.  A form of the Hermite-Biehler Theorem is proved for polynomials whose zeros lie in a strip (Lemma \ref{real-strip-lemma}), yielding an equivalent condition for a linear operator to be a real strip preserver (Proposition \ref{real-strip-h-equiv}).  The section ends with the proofs of Theorems \ref{stab-strip} and \ref{integral-shrink}.  

Let a linear operator $T:\RR_n[z] \rightarrow \RR[z]$ be extended to $\CC_n[z]=\RR_n[z]\oplus i\RR_n[z]$ by $T(u+iv)=T(u)+iT(v)$ for $u,v\in\RR[z]$.  While conditions \eqref{H1-stability} and \eqref{H2-stability} of Theorem \ref{strip} are sufficient for $T:\RR_n[z] \rightarrow \RR[z]$ to  preserve $H_1 \cup H_2$-stability, they are not necessary, as the following example demonstrates.  

\begin{example}\label{Rz-example}
The operator $T_1=e^{-3D^2/8}$ maps $\SS_1(\RR)$ to $\SS_{1/2}(\RR)$ by Theorem \ref{gaussian-diff}.  Let $T_2$ be the scale transformation $z\mapsto\frac{1}{2}z$, which maps $\SS_{1/2}$ to $\SS_1$.  The composition $T=T_2\circ T_1$  then preserves $\SS_1(\RR)$, but not $\SS_1$ (consider $T(z-i) =z/2 - i$). 
\end{example}

\subsection{A class of entire functions}
Throughout the rest of the section, let $f^*(z):=\overline{f(\bar{z})}$, and define the following class of functions for real $\mu\ge 0$,  
$$
\HD = \{ f ~:~ |f(z)|>|f^*(z)|~ \text{for}~ \Im~z>\mu,~\text{and}~f~\text{is entire} \}.
$$
Note that for $f\in\HD$, $|f(z)|<|f^*(z)|,$ whenever $\Im~z<-\mu.$  A function in $\HD$ may be of infinite order, although in subsequent investigations it may be beneficial to consider the subclass of $\HD$ consisting only of those functions with finite order.  We require the lemma of de Bruijn which is used to prove Theorem \ref{debruijn-jensen}, and state it here in a modified form with the notation just defined.

\begin{lemma}[{\cite[Lemma 1]{D}}]\label{debruijn-lemma}
If $p(z)\in\SS_\delta(\RR)$, $\delta\ge0$, then for $\lambda>0$, 
\[
p(z+i\lambda)\in\HD,
\]
where $\mu=\sqrt{\max\{\delta^2-\lambda^2,0\}}$. 
\end{lemma}

\begin{lemma}\label{h-star-conditions}
Let $f=g+ih$ be an entire function,  where $g$ and $h$ are real entire functions.  The following are equivalent. 
\renewcommand\theenumi{\roman{enumi}}
\begin{enumerate}\itemsep 2pt
\item $f\in\HD.$ 

\item $ wf(z)+f^*(z) \neq 0,\text{ whenever }|w|\ge1 \text{~and~} \Im~z>\mu.$

\item $ g(z)+w h(z) \neq 0,\text{ whenever }\Im~w\ge0 \text{~and~} \Im~z>\mu.$

\end{enumerate}
Furthermore, if $f\in\HD$, then $f(z)$, $g(z)$, and $h(z)$ are non-zero for $\Im~z>\mu$.
\end{lemma}

\begin{proof}
To show condition $(i)$ is equivalent to condition $(ii)$, note that $f\in\HD$ if and only if
$$ 
|f^*(z)|/|f(z)|<1,~\Im~z>\mu.
$$ 
This is in turn equivalent to 
$$ 
wf(z)+f^*(z) \neq 0,\text{ whenever }|w|\ge1,~\Im~z>\mu.
$$  
Condition $(ii)$ is equivalent to condition $(iii)$ by virtue of the M\"obius transformation $w\mapsto (w+i)/(w-i)$ which maps the upper half-plane to the exterior of the unit disk, yielding the condition 
$$ 
w(f+f^*)+i(f-f^*)\neq 0,\text{ whenever }\Im~w\ge0, \Im~z>\mu.
$$
Either setting $w=0$ in (iii), or multiplying by $1/w$ and applying Hurwitz's theorem as $|w|\to\infty$, shows that $f(z)$, $g(z)$, and $h(z)$ are non-zero for $\Im~z>\mu$. Note that $\HD$ does not contain complex constant multiples of real entire functions, including the identically zero function.  This prevents the cases $h\equiv 0$ in the application of Hurwitz's theorem above.
\end{proof}

\begin{remark}
The conditions in Lemma \ref{h-star-conditions} are similar to those for majorization in \cite[Ch. IX]{Levin}.  For $\mu=0$, the class of functions in $\HD$ is the class of de Branges functions \cite{deBranges59}, called the Hermite-Biehler class of functions in \cite{KalW}, and includes the P\'olya class, as defined by L.~de~Branges \cite[p. 35]{deBranges}.
\end{remark}
\noindent
Let $\HDD$ be the closure of $\HD$ under uniform limits on compact subsets of $\CC$.  The following is an extension of Lemma \ref{h-star-conditions} to functions in $\HDD$.

\begin{lemma}\label{h-star-bar-conditions}
Let $f=g+ih$ be an entire function,  where $g$ and $h$ are real entire functions.  The following are equivalent.
\renewcommand\theenumi{\roman{enumi}}
\begin{enumerate} \itemsep 2pt
\item $f\in\HDD.$
\item Either $f\equiv 0$, or for  $|w|>1$, $\Im~z>\mu$,
$$ wf(z) +f^*(z) \neq 0.$$  
\item Either $f\in\HD$, or $f=c \varphi$, where $c\in\CC$, and $\varphi$ is a real entire function such that $\varphi(z)\neq 0$ whenever $\Im~z >\mu$.
\item Either $f\equiv 0$, or for $\Im~w>0$, $\Im~z>\mu$,
$$g(z)+wh(z) \neq 0.$$
\end{enumerate}
Furthermore, $f(z)$, $g(z)$, and $h(z)$ are either non-zero for $\Im~z>\mu$ or identically zero.
\end{lemma}

\begin{proof}
We first show (i)$\Leftrightarrow$(ii). 
Assume (i) and let $\{f_n\}_{n=1}^\infty$ be a sequence of entire functions in $\HD$ which converges uniformly on compact subsets of $\CC$ to a function $f$.  Then by Lemma \ref{h-star-conditions}, 
\begin{equation}\label{fneq}
y_n(z,w) := wf_n(z) +f^*_n(z) \neq 0, \qquad |w|>1,~\Im~z>\mu,
\end{equation}
for each $n\in\NN$, and consequently by Hurwitz's Theorem, either the inequality in \eqref{fneq} holds in the limit $n\to\infty$ or we obtain the identically zero function, which establishes (ii). If (ii) holds and $f$ is not identically zero, then
$$ wf(z) +f^*(z) \neq 0, \qquad |w|>1,~\Im~z>\mu,$$
and $f(z)\neq 0$ for $\Im~z>\mu$ by Hurwitz's theorem.  It follows that $|f^*(z)|/|f(z)|\le 1$ for $\Im~z>\mu$.
Let $u_n(z) = (1+iz/n)f(z)$ define a sequence of entire functions $\{u_n\}_{n=1}^\infty$ which converge to $f$ uniformly on compact subsets of $\CC$. For each $n\in\NN$,
$$\frac{|u^*_n(z)|}{|u_n(z)|} = \left|\frac{1-iz/n}{1+iz/n}\right|\left|\frac{f^*(z)}{f(z)}\right| < 1, \qquad  \Im~z>\mu.$$ 
Therefore, $wu_n(z)+u_n^*(z) \neq 0$, for $|w|\ge1$, $\Im~z>\mu.$  By Lemma \ref{h-star-conditions}, $u_n\in\HD$ for each $n$, thus $f\in\HDD.$ 

Conditions (ii) and (iv) are seen to be equivalent by applying the same M\"obius transformation as in the proof of Lemma \ref{h-star-conditions}. Note that if (ii) and (iv)  are satisfied, then $f(z)$, $g(z)$, and $h(z)$ are either non-zero for $\Im~z>\mu$ or identically zero by Hurwitz's theorem.

It is easy to see that (iii)$\Rightarrow$(i)$\wedge$(ii). Moreover if (ii) is satisfied and $0 \not \equiv f \not \in \HD$, then $f^*(z)=-wf(z)$ for some $|w|=1$. It follows that $f=c \varphi$, where $c\in\CC$, and $\varphi$ is a real entire function. Hence, $\varphi$ is non-zero for $\Im~z>\mu$ as observed in the previous paragraph.  This establishes (ii)$\Rightarrow$(iii).
\end{proof}

 Let $\HDD(\RR)$ be the subclass of \emph{real} entire functions in $\HDD$, and note that by Proposition \ref{h-star-bar-conditions}, if $f(z)\in\HDD(\RR)$, then $f(z)$ is non-zero for $|\Im~z|>\mu$. For entire functions $f$ and $g$ we define the 
$\mu$-\emph{Wronskian} by 
\begin{align}\label{wren}
W_\mu[f,g](z):\!\!&=\frac {\sin(\mu D)f}{\mu}  \cdot \cos(\mu D)g  - \cos(\mu D)f \cdot \frac {\sin(\mu D)g}{\mu} \nonumber \\ 
&= \frac 1 {2\mu i} ( f(z+i\mu)g(z-i\mu)-f(z-i\mu)g(z+i\mu)).
\end{align}
Note that $W_\mu[f,g]$ reduces to the usual Wronskian $W_0[f,g]:=f'g-fg'$ when $\mu \rightarrow 0$. 
\begin{lemma}\label{wronskians}
Let $\mu \geq 0$ and $f=g+ih$, where $g$ and $h$ are real entire functions. 
\renewcommand\theenumi{\roman{enumi}}
\begin{enumerate}
\item $f \in \HD$ if and only if  $W_\lambda[h,g](x)<0$ for all $\lambda > \mu$ and $x \in \RR$. 
\item If $f \in \HDD$, then 
$W_\mu[h,g](x) \leq 0$ for all $x \in \RR$. \label{f-in-hdd-wron}
\item If $f^* \in \HDD$, then 
$W_\mu[h,g](x) \geq 0$ for all $x \in \RR$. 
\item If $f \in \HDD$, then $W_\mu[h,g] \equiv 0$ if and only if $f \in \HDD \setminus \HD$.
\end{enumerate}
\end{lemma}
\begin{proof}
Let $\mu \geq 0$ and $f=g+ih$.  By Lemma \ref{h-star-conditions} (iii), $f \in \HD$ if and only if for all $\lambda > \mu$, 
$$
g(x+i\lambda) + wh(x+i\lambda)=0 \ \ \mbox{ implies } \ \ \Im~w <0.
$$
If $h(x+i\lambda)=0$, then $W_\lambda[h,g]=0$ and $f \notin \HD$. If $h(x+i\lambda) \neq 0$, 
then  
$$
\Im~w = -\Im~\frac {g(x+i\lambda)}{h(x+i\lambda)}  = \lambda \frac {W_\lambda[h,g](x)} {|h(x+i\lambda)|^2}, 
$$
which proves (i). 

Statements (ii) and (iii)  follow by continuity and the anti-symmetry of $W_\mu$. 

For (iv), assume $f \in \HDD$ and $W_\mu[h,g] \equiv 0$, where $\mu > 0$. If $h \equiv 0$, then $f \in \HDD \setminus \HD$ by Lemma \ref{h-star-bar-conditions}  (iii). If $h \not \equiv 0$, then by \eqref{wren},  
$g(z)/h(z)=g(z+2\mu i)/h(z+2\mu i)$ for all $z \in \CC$ such that $h(z) \neq 0$.   If $g(z)/h(z)$ is not a constant, then there is a $w\in H$ such that $g(z)/h(z) +w$ is periodic and somewhere $0$.  Hence the entire function $g(z)+ wh(z)= h(z)(g(z)/h(z)+w)$ fails to satisfy Lemma \ref{h-star-bar-conditions} part (iv), unless $g(z)/h(z)$ is a constant. Now  $g(z)/h(z)$ is a constant precisely  when $f \in \HDD \setminus \HD$. For $\mu=0$, the condition is easily seen to imply that $g$ and $h$ are constant multiples of each other. 
\end{proof}
\noindent
Note that by Lemma \ref{wronskians} \eqref{f-in-hdd-wron}, the class of functions associated to a fixed $h\in\HDD(\RR)$ by
\[
\{f~:~f+ih\in\HDD, \text{ $f$ is a real entire function}\}
\] 
forms a cone (it is closed under convex sums).

A real linear space of entire functions, all of which are non-zero on the same prescribed region, can have dimension at most two by Lemma \ref{spaces}.  Consequently, Lemma \ref{real-strip-lemma} is an analog to the Hermite-Biehler Theorem for polynomials in $\HDD$, and characterizes linear spaces of real polynomials whose zeros lie in a strip. 

\begin{lemma}\label{real-strip-lemma}
Let $f$ be an entire function. Then 
$f\in\HD \cup \HD^*$ if and only if $f=g+ih$, where $g$ and $h$ are real entire functions such that 
\begin{equation}\label{real-pair} \alpha g+\beta h\in\HDD(\RR)\setminus \{0\} \text{~for all~} (\alpha,\beta)\in\RR^2\setminus \{(0,0)\}.
\end{equation}
Moreover $f \in \HD$ ($f \in \HD^*$) if and only if $f$ satisfies \eqref{real-pair} and $W_\mu[h,g](x) < 0$ \textup{(}$W_\mu[h,g](x) >0$\textup{)} for some $x \in \RR$. 
\end{lemma}

\begin{proof}
If $f=g+ih\in\HD \cup \HD^*$, then \eqref{real-pair} follows from Lemma \ref{h-star-conditions} (iii) and the statement just below it. 

For the converse consider the entire function in two variables $g(z)+wh(z)$. It is non-vanishing whenever $\Im~z>\mu$ and $w \in \RR$. By Lemma \ref{one}, $g(z)+wh(z)$ is either non-vanishing on $\{z :  \Im~z>\mu\} \times \{w :  \Im~w \geq 0\}$, or on $\{z :  \Im~z>\mu\} \times \{w :  \Im~w \leq 0\}$. Hence either $f \in \HD$ or $f^* \in \HD$ by Lemma \ref{h-star-conditions}. 

If $f \in\HD \cup \HD^*$, then $W_\mu[h,g](x) \neq 0$ for some $x \in \RR$ by Lemma \ref{wronskians} (4). Hence, the last statement follows from Lemma \ref{wronskians}. 
\end{proof}

\begin{remark}
When $\mu =0$, Lemma \ref{real-strip-lemma} reduces to the ordinary Hermite-Biehler Theorem, with the interlacing condition replaced by an equivalent condition on the span of $g$ and $h$ \cite[pp. 12--13]{O}.
\end{remark}

Let $\PPD$ be the set of polynomials in $\HDD$ defined by
\[
\PPD=\HDD\cap\CC[z].
\]
Note that $\PPD$ is a subset of the polynomials that are stable on the half-plane $\Im~z>\mu$. 
By Lemma \ref{real-strip-lemma}, $\PPD$ consists of precisely those polynomials in $\PP(\{z\,:\,\Im~z>\mu\})$ which can be written as $g+ih$, where the real polynomials $g$ and $h$ form a Hermite-Biehler type pair (although $g,h$ may have non-real zeros).  By the classical Hermite-Biehler theorem  \cite[pp. 12--13]{O}, a stable polynomial will satisfy this condition, and thus the following strict inclusions hold:
\begin{equation}
\PP(H) \subset \PPD \subset \PP(\{z\,:\,\Im~z>\mu\}). \label{stab-h-dstab} 
\end{equation}
We now prove that the set of polynomials in $\PPD$ of degree less than $n\in\NN$, and its interior, are connected.  The proof of this relies on deforming polynomials in $\PPD$ into stable polynomials, and using the the connectedness of the set of stable polynomials of degree less than $n$.  Note that any stable polynomial of degree $d<n$ may be obtained as a limit of stable polynomials of degree $n$ by letting $n-d$ zeros tend to infinity in the lower half-plane.

\begin{lemma}\label{connected}
Both $\PPD \cap \CC_n[z]$ and its interior, $\Int(\PPD \cap \CC_n[z])$, are connected. Moreover, each polynomial in $\PPD \cap \CC_n[z]$ is the limit of polynomials in $\Int(\PPD \cap \CC_n[z])$. 
\end{lemma}

\begin{proof}
We claim that for each $0 \leq t \leq \mu$, 
\[
\cos(tD): \PPD \to \PP_{\delta(t)}, 
\]
where $\delta(t)=\sqrt{\mu^2-t^2}$. This will prove the first part of the lemma by the following arguments. Note that 
$\PP_0 \cap  \CC_n[z]$ is the set of stable polynomials of degree at most $n$, and $\Int(\PP_0 \cap  \CC_n[z])$ is the set of stable polynomials of degree at most $n$ with no real zeros, respectively. These sets are obviously connected. Hence, given the claim, this proves that $\PP_\mu \cap  \CC_n[z]$ is connected, since we may first continuously deform a polynomial $p \in \PPD$ to a stable polynomial using $\cos(tD)$, $0\leq t \leq \mu$, and we may connect any two stable polynomials. Since $\cos(tD): \CC_n[z] \rightarrow \CC_n[z]$ is invertible and continuous we have 
$\cos(tD): \Int(\PP_\mu \cap  \CC_n[z]) \rightarrow \Int(\PP_{\delta(t)} \cap \CC_n[z])$, which implies that we may connect any two polynomials in $\Int(\PPD \cap \CC_n[z])$ as just described.  

Note that the claim is true for real polynomials in $\PPD$ by Theorem \ref{debruijn-jensen}. For complex polynomials in 
$\PPD$ we argue as follows. By continuity it suffices to consider $p=q+ir \in \HD \cap  \CC_n[z]$. Then $\alpha q+\beta r \in \PP_\mu \cap \RR_n[z]$ for all $\alpha, \beta  \in \RR$, by Lemma \ref{real-strip-lemma}. Hence, 
$\alpha \cos(tD)q+\beta \cos(tD)r \in \PP_{\delta(t)} \cap \RR_n[z]$ for all $\alpha, \beta  \in \RR$, and thus $\cos(tD)p \in \PP_{\delta(t)}\cup \PP_{\delta(t)}^*$ for all $0\leq t \leq \mu$ by the same lemma.  If the claim is false, then 
$\cos(tD)p \in \PP_{\delta(t)}^*$ for some $0\leq t \leq \mu$. However, for small $t>0$,  
$W_{\delta(t)}[\cos(tD)q,  \cos(tD)r](x)<0$ for some $x \in \RR$ by continuity and Lemma \ref{real-strip-lemma}, and therefore 
$\cos(tD)p \in \PP_{\delta(t)}$ for sufficiently small $t>0$ by the same lemma. It follows by continuity that there is a $0<t \leq \mu$ such that $\cos(tD)p \in \PP_{\delta(t)}\cap \PP_{\delta(t)}^*$, and thus $\cos(tD)p$ is a complex constant multiple of a real polynomial by Lemma \ref{h-star-bar-conditions} (iii); that is, 
$$
\cos(tD) (\alpha q+\beta r) \equiv 0 \mbox{ for some } (\alpha, \beta) \in \RR^2\setminus \{(0,0)\}.
$$
Since $\cos(tD)$ is invertible this implies that $\alpha q+\beta r \equiv 0$, which is a contradiction.
This proves the claim.

If $\mu=0$, the second statement is clear, so let $\mu >0$. If $p \in \mathscr{D}_\lambda \cap \CC_n[z]$ for some $0 \leq \lambda < \mu$, then 
\begin{equation}\label{iint}
p(z)+wp^*(z) \neq 0, \mbox{ whenever } |w| \leq 1 \mbox{ and } \Im ~z \geq \mu
\end{equation}
by Lemma \ref{h-star-conditions} (ii). 
For each fixed $|w_0|\le 1$, the set of polynomials which satisfy $p(z)+w_0p^*(z)\neq 0$ is open, and since $\{w: |w|\le 1\}$ is compact, the set of polynomials satisfying \eqref{iint} is open. 
It follows that $\mathscr{D}_\lambda \cap \CC_n[z] \subseteq \Int(\PPD \cap \CC_n[z])$. Hence if, $p \in \mathscr{D}_\mu \cap \CC_n[z]$, then $\cos(\epsilon D)p \in \Int(\PP_\mu \cap \CC_n[z])$ for each $\epsilon >0$.  Otherwise, if $p \in \PP_\mu \cap \CC_n[z]$ is not in $\mathscr{D}_\mu \cap \CC_n[z]$, then $p$ is a constant multiple of a real polynomial and by Lemma \ref{debruijn-lemma}, $p(z+i\epsilon) \in \Int(\PP_\mu \cap \CC_n[z])$ for each $\epsilon >0$ by the same arguments (unless $p$ is a constant, which is clearly in $\overline{\Int(\PP_\mu \cap \CC_n[z])}$). 
\end{proof}

The polynomials of degree $1$ in $\HDD$ are simply those with a zero in the closed lower half-plane, and descriptions of the degree $2$ polynomials in $\HDD$ using the location of zeros are more complicated.
The structure of $\HDD$ appears non-trivial and open to investigation.

\subsection{Real strip preservers}

Recall that $H$ is defined to be the open upper half-plane.  We use the notation that for a set of polynomials $\SS$, 
\[
\SS^* = \{p^*:p\in \SS\}.
\]
An operator which preserves real zeroed polynomials must be either stability preserving or \emph{stability reversing} \cite{BB1}, meaning either $T(\PP(H))\subseteq \PP(H)\cup\{0\}$ or $T(\PP(H))\subseteq \PP^*(H)\cup\{0\}$.  A salient feature of the class $\PPD$ is that preserving the strip of width $\mu$ is equivalent to either preserving or reversing $\PPD$.  This yields the following ``characterization''
of real strip preservers.

\begin{proposition}\label{real-strip-h-equiv}
Let $T:\RR_n[z]\to\RR[z]$ be a linear operator. \\
Then $T:\SS_\mu\cap\RR_n[z]\to\SS_\mu(\RR)\cup\{0\}$ if and only if either
\renewcommand\theenumi{\roman{enumi}}
\begin{enumerate}\itemsep 2pt
\item $\dim_\RR(T)\le 2$ and $T$ has the form
$$T(p)=\alpha(p)Q + \beta(p)R,$$
where $\alpha,\beta:\RR_n[z]\to\RR$ are linear functionals, and $Q+iR\in\PPD$, or   
\item either $T(\PPD\cap \CC_n[z])\subseteq\PPD$ or $T(\PPD\cap\CC_n[z])\subseteq\PPD^*$.  
\end{enumerate}
\end{proposition}

\begin{proof}Consider first the case where $\dim_\RR(T)> 2$.
By Lemma \ref{h-star-bar-conditions}, the set of real polynomials in $\PPD$ is $S_\mu(\RR)\cup\{0\}$. Thus, if $T$ is a real linear operator such that (ii) holds, then $T:\SS_\mu\cap\RR_n[z]\to\SS_\mu(\RR)\cup\{0\}$.  For the converse, suppose $T(\SS_\mu\cap\RR_n[z])\subseteq\SS_\mu(\RR)$.  Then for $p=q+ir\in\PP_\mu$ of degree at most $n$, $T(q)+wT(r)\in\PP_\mu$ for all real $w$, by Lemma \ref{real-strip-lemma}.  
By Lemma \ref{real-strip-lemma} again, $T(p) = T(q)+iT(r)$ is either in $\PPD$ or $\PPD^*$.  Hence, $T(\PPD \cap \CC_n[z])\subseteq\PPD\cup\PPD^*$, and we claim that $T(\PPD \cap \CC_n[z])\subseteq\PPD$ or $T(\PPD \cap \CC_n[z])\subseteq\PPD^*$.  It is sufficient to show that $T(\Int(\PPD \cap \CC_n[z]))\subseteq \PPD$ or $T(\Int(\PPD \cap \CC_n[z]))\subseteq \PPD^*$, by Lemma \ref{connected}.  
Suppose that there are $p_1, p_2\in\Int(\PPD \cap \CC_n[z])$ with $T(p_1)\in \PPD$ and $T(p_2)\in \PPD^*$.  Then by Lemma \ref{connected}, there is a polynomial  $p \in\Int(\PPD \cap \CC_n[z])$ such that $T(p)\in \PPD\cap \PPD^*$. By Lemma \ref{h-star-bar-conditions},  the space $\PPD\cap \PPD^*$ is precisely the set of  complex constant multiples of polynomials in $\SS_\mu(\RR)$. By multiplying $p$ by a suitable complex constant we may thus assume that $iT(p) \in \PPD \cap \RR[z]$, and thus $p=q+ir$ where $T(q) \equiv 0$. Now let $h \in \RR_n[z]$ be arbitrary. Then, since $p \in\Int(\PPD \cap \CC_n[z])$, there is an $\epsilon >0$ such that $p+\epsilon h \in \PPD$. Hence 
$\epsilon T(h)+iT(r) = T(p+\epsilon h) \in \PPD$. By Lemma \ref{h-star-bar-conditions} again this implies that $T(h)$ is either identically zero or in $\SS_\mu$. We have proved that each element of the linear space 
$T(\RR_n[z])$ is either identically zero or in $\SS_\mu$. By Lemma \ref{spaces}, $T$ can have rank at most $2$, which corresponds to the first case.  Thus if $\dim_\RR T >2$, $T[\Int(\PPD\cap\CC_n[z])]\subseteq\PPD$ or $T[\Int(\PPD\cap\CC_n[z])]\subseteq\PPD^*$, and case (ii) holds.

If $\dim_\RR(T)\le 2$, then $T$ must have the form in (i) by Lemma \ref{real-strip-lemma}.
\end{proof}

In the $\mu=0$ case, Proposition \ref{real-strip-h-equiv} provides the characterization of real stability preservers \cite{BB1} as stability preserving or reversing.  If $T$ preserves $S_\mu(\RR)$ for all degrees, then Proposition \ref{real-strip-h-equiv} implies that either $T(\PPD\cap \CC_n[z])\subseteq\PPD$ or $T(\PPD\cap\CC_n[z])\subseteq\PPD^*$ for all $n\in\NN$, yielding the following corollary. 

\begin{corollary}\label{soft-strip}
The linear operator $T:\RR[z]\to\RR[z]$ preserves $\SS_\mu(\RR)$ if and only if either
\renewcommand\theenumi{\roman{enumi}}
\begin{enumerate}\itemsep 2pt
\item $\dim_\RR(T)\le 2$ and $T$ has the form
$$T(p)=\alpha(p)Q + \beta(p)R,$$
where $\alpha,\beta:\RR[z]\to\RR$ are linear functionals, and $Q+iR\in\PPD$, or   
\item either $T(\PPD)\subseteq\PPD$ or $T(\PPD)\subseteq\PPD^*$.
\end{enumerate}
\end{corollary}

From Corollary \ref{soft-strip}, one immediately obtains the following necessary condition for a real strip preserver.

\begin{proposition}\label{real-strip}
Let $T:\RR[z]\to\RR[z]$, $H_1=\{z\,:\,\Im~z > \mu\}$, and $H_2=H_1^*$.
If $T$ preserves $\SS_\mu(\RR)$, then either 
$$T(\PP(H))\subseteq\PP(H_1)\cup\{0\} ~\text{~and~}~ T(\PP(H^*))\subseteq\PP(H_2)\cup\{0\},$$
 or 
$$T(\PP(H))\subseteq\PP(H_2)\cup\{0\} ~\text{~and~}~ T(\PP(H^*))\subseteq\PP(H_1)\cup\{0\}.$$
\end{proposition}

\begin{proof}
By symmetry, it is only necessary to consider $T(\PP(H))$.  Suppose $h=q+ir\in\PP(H)$.  By Lemma \ref{addz}, $q+\lambda r\in\PP(H)\cap\RR[z]$ for all $\lambda\in\RR$, and in particular $q+\lambda r\in\PP(H_1\cup H_2)$.  Then either  $T(q+\lambda r) \equiv 0$ or 
$$T(q+\lambda r)  = T(q)+\lambda T(r) \in \PP(H_1\cup H_2) \text{ for all } \lambda\in\RR.$$
In the latter case, $T(q)+ iT(r)\in\PP(H_1)\cup\PP(H_2)$ by Lemma \ref{real-strip-lemma}, and Corollary \ref{soft-strip} restricts the images of $\PP(H)$ and $\PP(H^*)$ to those stated, since $T(\PP(H))\not\subseteq\PP(H_2)\cup\{0\}$ when $T(\PPD)\subseteq\PPD$, and  $T(\PP(H))\not\subseteq\PP(H_1)\cup\{0\}$ when $T(\PPD)\subseteq\PPD^*$.  
\end{proof}

The condition in Proposition \ref{real-strip} is not sufficient. For example, any scale transformation, including a dilation which moves zeros outside the strip, will satisfy these conditions.  For the special case of a differential operator with constant coefficients, the characterization (Corollary \ref{real-strip-constant-coeffs}) is simple, and coincides with Corollary \ref{strip-diff-op}.  By the theory of real stability, we now obtain sufficient conditions to preserve $\PPD$.

\begin{theorem}\label{P-suff-alg}
Let $T : \CC_n[z] \rightarrow \CC[z]$ be a linear operator with algebraic symbol $G_T(z,\zeta)=A(z,\zeta)+iB(z,\zeta)$, where $A,B\in\RR[z,\zeta]$. If 
\begin{equation}\label{zzw}
A(z,\zeta)+wB(z,\zeta) \neq 0 \mbox{ whenever } \Im~z >\mu, \Im~\zeta >-\mu, \mbox{ and }  \Im~w >0,
\end{equation}
then $T : \PP_\mu \cap  \CC_n[z]  \rightarrow \PP_\mu$. 
\end{theorem}

\begin{proof}
Write $T=R+iJ$, where $R,J : \RR_n[z] \rightarrow \RR[z]$. Then 
$$
T(f+ig)= R(f)-J(g)+i(J(f)+R(g)).
$$ 
Define a linear operator $\widetilde{T}: \RR_{(n,1)}[z,w] \rightarrow \RR_{(n,1)}[z,w]$ by 
$$
\widetilde{T}(f+wg)= R(f)-J(g)+w(J(f)+R(g)).
$$ 
By way of Lemma \ref{h-star-bar-conditions}, the result follows if we prove that \eqref{zzw} implies that $\widetilde{T}$ preserves 
stability on $\{z : \Im~z >\mu\} \times \{w: \Im~w >0\}$. By \cite[Theorem 2.1]{BB2}, $\widetilde{T}$ preserves 
such stability if 
$
G_{\widetilde{T}}(z,\zeta, w, v)= \widetilde{T}[(z+\zeta)^n(w+v)]
$
is non-zero whenever $\Im~z >\mu, \Im~\zeta >-\mu, \Im~w >0$, and $\Im~v>0$. Since 
$$
G_{\widetilde{T}} = (w+v)\left( A(z,\zeta)+ \left( \frac 1 {1/w+1/v}-\frac 1 {w+v} \right)B(z,\zeta) \right)
$$
and 
$$
\Im~w>0 \mbox{ and } \Im~v>0 \mbox{ implies } \Im\left( \frac 1 {1/w+1/v}-\frac 1 {w+v} \right) >0, 
$$
the proof follows. 
\end{proof}

With a proof almost identical to that of Theorem \ref{P-suff-alg}, but using the transcendental characterization of stability preservers \cite[Theorem 2.3]{BB2} we get the following sufficient condition for preserving $\PP_\mu$. 

\begin{theorem}\label{P-suff-tran}
Let $T : \CC[z] \rightarrow \CC[z]$ be a linear operator with symbol $\overline{G}_T(z,\zeta)=A(z,\zeta)+iB(z,\zeta)$, where $A,B\in\RR[z][[\zeta]]$, and fix $\mu\ge0$. If 
\begin{equation}\label{zzwt}
e^{i\mu \zeta}(A(z+i\mu,\zeta)+wB(z+i\mu,\zeta))
\end{equation}
defines an entire function in $\PPP(H\times H \times H)$,  
then $T : \PP_\mu  \rightarrow \PP_\mu$. 
\end{theorem}

\begin{remark} 
The sufficient conditions to preserve $\PP_\mu$ in Theorems \ref{P-suff-alg} and \ref{P-suff-tran} are not necessary.  For example, the transcendental symbol for the operator $T$ in Example \ref{Rz-example} is $\overline{G}_T(z,\zeta)=e^{-3\zeta^2/8-z\zeta/2}=A(z,\zeta)$, and
\[
e^{i\mu \zeta}(A(z+i\mu,\zeta)+wB(z+i\mu,\zeta)) = e^{-3\zeta^2/8-z\zeta/2+i\mu\zeta/2} \not\in \PPP(H\times H \times H).
\] 
\end{remark}

\begin{theorem}\label{constant-coefficients}
Let $f(z)$ be a formal power series with complex coefficients. Then $f(D)$ preserves $\PP_\mu$ if and only if $f$ is an entire function in $\PPP(H^*)$.
\end{theorem}

\begin{proof}
Note that in the case $\mu =0$, the theorem reduces to the well-known result that $f(D)$ is a stability preserver if and only if $f\in\PPP(H^*)$  (see \cite{BB1}). 

Let $T=f(D)=h(D)+ig(D)$, where $h$ and $g$ are real power series, and note that $\overline{G}_T(z,\zeta)= e^{-z\zeta}f(-\zeta)$. Hence \eqref{zzwt} simplifies to 
$e^{-z\zeta}(h(-\zeta)+wg(-\zeta))$. 
If $f \in \PPP(H^*)$, then $f(-\zeta)=h(-\zeta)+ig(-\zeta) \in \PPP(H)$. By Remark \ref{trans-addz}, $h(-\zeta)+wg(-\zeta) \in \PPP(H\times H)$. Hence sufficiency follows from Theorem \ref{P-suff-tran} and that $e^{-z\zeta} \in \PPP(H\times H)$. 

Suppose $T$ preserves $\PPD$. If $p \in \SS_\mu(\RR)\subseteq \PPD$, then  
$h(D)p+ig(D)p \in \PP_\mu$. Fix $\alpha \in \RR$, and consider the operator $T_\alpha= h(D)+\alpha g(D)$. It follows from Lemma \ref{h-star-bar-conditions} that $T_\alpha : \SS_\mu(\RR) \rightarrow \SS_\mu(\RR)$. Thus by Proposition \ref{real-strip-h-equiv}, $T_\alpha : \PP_\mu \rightarrow \PP_\mu$ or $T_\alpha : \PP_\mu \rightarrow \PP_\mu^*$. If $T_\alpha : \PP_\mu \rightarrow \PP_\mu$, then $e^{i\mu D}T_\alpha : \PP_0 \rightarrow \PP_0$, since the translation $e^{i\mu D}$ ensures that condition (iv) in Lemma \ref{h-star-bar-conditions} is satisfied (for $\mu=0$ in Lemma \ref{h-star-bar-conditions}). It follows that  $e^{i\mu D}T_\alpha$ preserves stability, since $\PP_0=\PP(H)$, and therefore by the characterization of stability preservers \cite{BB1},
$e^{-i\mu z}(h(-z)+\alpha g(-z))$ is in $\PPP(H)$. 
Since $h+\alpha g$ is a real entire function we may then deduce that also $h+\alpha g \in \PPP(H)\cap\RR[[z]]$ (by \cite[Chapter VII, Theorem 7]{Levin}). If $T_\alpha : \PP_\mu \rightarrow \PP_\mu^*$, then $R\circ e^{-i\mu D} \circ T_\alpha : \PP_0 \rightarrow \PP_0$ where $R(p(z))=p(-z)$. By the characterization of stability preservers in \cite{BB1}, the symbol of $R\circ e^{-i\mu D} \circ T_\alpha$ is in $\PPP(H \times H)$. The symbol is $e^{z\zeta+i\mu \zeta}(h(-\zeta)+\alpha g(-\zeta))$, which is not in 
 $\PPP(H \times H)$ unless $h(-\zeta)+\alpha g(-\zeta) \equiv 0$. In any case we have proved that $h(z)+\alpha g(z) \in \PPP(H)$ for each $\alpha \in \RR$.  This implies by Remark \ref{trans-addz}, that either $f=h+ig \in \PPP(H^*)$ or $f^*=h-ig \in \PPP(H^*)$. 
 
To finish the proof we need to prove that $f \in \PPP(H^*)$. If $f^* \in \PPP(H^*)$, then  
$f^*(D) : \PPD \to \PPD$ (by the converse direction) and $f(D) : \PPD \to \PPD$ (by assumption). Since 
$T : \PPD \to \PPD$ if and only if $T^* : \PPD^*\to \PPD^*$, we have $f(D) : \PPD\cap \PPD^* \to \PPD\cap \PPD^*$. 
 By Lemma \ref{h-star-bar-conditions} (iii), 
\[ 
\PPD\cap \PPD^*= \{ch~:~h\in\SS_\mu(\RR), c\in\CC\}.  
\]
Since $f(D)$ maps complex multiples of real polynomials to polynomials of the same type, $f$ must be a complex multiple of a real entire function, and hence both $f^*\in\PPP(H^*)$ and $f\in\PPP(H^*)$.

\end{proof}

\begin{corollary}\label{real-strip-constant-coeffs}
Let $f(z)$ be a real formal power series. Then $f(D)$ preserves $\SS_\mu(\RR)$ if and only if $f\in\LP$.
\end{corollary}

\begin{proof}
Fix $\mu\ge0$.  By Corollary \ref{soft-strip}, $f(D)$ preserves $\SS_\mu(\RR)$ if and only if either $f(D)\PPD\subseteq\PPD$ or $f(D)\PPD\subseteq\PPD^*$.  From Theorem \ref{constant-coefficients}, this can only happen if $f(z)\in\RR[[z]]\cap\PPP(H^*)$ or $f(z)\in\RR[[z]]\cap\PPP(H)$, or equivalently $f\in\LP$.
\end{proof}

We now prove Theorem \ref{stab-strip}, which extends de Bruijn's theorem. 

\begin{proof}[Proof of Theorem \ref{stab-strip}]
Assume $T=e^{i\lambda D}h(D)$ where $h(w)\in\PPP(H^*)$.  Now 
$$ (T+T^*)p(z) =  h(D)R(z) +h^*(D)R^*(z),$$
where $R(z)=e^{i\lambda D}p(z) = p(z+i\lambda)$.  Hence $R\in\PP_\delta$ with $\delta=\max(0,\sqrt{\mu^2-\lambda^2})$ by Lemma \ref{debruijn-lemma}. Since the operator $h(D)$ preserves the set $\PP_\delta$ by Theorem \ref{constant-coefficients}, the zeros of $(T+T^*)p(z)=h(D)R(z) +h^*(D)R^*(z)$ will lie in the strip of width $\delta$ by Lemma \ref{h-star-bar-conditions} (with $f=h(D)R(z)$, $f+f^*\in\SS_\delta(\RR)$).  
\end{proof}

Recall that $\PPP(\SS)$ is the closure of the set of polynomials $\PP(\SS)$ under uniform limits on compact sets in $\CC$.  An entire function $f\in\PPP(\{z\,:\, \Re~z>0\})$ is said to be \emph{Hurwitz stable}.  We will need the following theorem of Enestr\"om--Kakeya, see \cite[p. 255]{RS}. 

\begin{theorem}[Enestr\"om--Kakeya]\label{E-K}
If the coefficients of  a polynomial are non-negative and non-decreasing, then 
all its zeros lie in the closed unit disk. 
\end{theorem} 

\begin{proof}[Proof of Theorem \ref{integral-shrink}]
Through scaling $z$, it is sufficient to prove the theorem when $\lambda=1$.  Let $y(z)=e^{i z/2}I(z)/2$, where
$$
I(z)=\int_0^1 e^{iz(t-1/2)}g(t) \;dt, 
$$
and thus $f(z)=y(z)+y^*(z)$, and $f(D)=e^{iD/2}I(D)/2 + e^{-iD/2}I^*(D)/2$.  
We will show $I(z)\in\PPP(H^*)$, whence the result follows by Theorem \ref{stab-strip}.  

Rotating $z\to -iz$, and shifting the integral by $t\mapsto t+1/2$, it is sufficient to prove the statement that 
$$
F(z)= \int_{-1/2}^{1/2} e^{zt} g(t+1/2) dt 
$$
is Hurwitz stable. Approximate the integral, $F(z)$, by $G_M(z)$, where
$$
(2M+1)G_M(z)= \sum_{k=-M}^M g\left(\frac{k}{2M}+\frac{1}{2}\right) \left(e^{z/2M}\right)^k = H_M\left(\frac{z}{2M}\right).
$$
Note that by the Enestr\"om--Kakeya theorem, $H_M(z)$ may be written as product of real factors of the form $R(z)=e^{-z/2}(e^z+r)$, where $r$ is a real number in the interval $[-1,1]$, and complex conjugate factors of the form 
$$
U(z)= e^{-z}(e^z+\zeta)(e^z+\bar{\zeta})= e^z+s^2e^{-z}+2t ,\quad \mbox{ where } |t|\leq s = e^{-\tau} \leq 1, 
$$
and $|\zeta|\le 1$.
Hence 
$$
2e^\tau U(z) = \cosh(z+\tau)+ t/s \qquad (\tau\le 0). 
$$
Note that if $U(z-\tau)$ is Hurwitz stable, then so is $U(z)$. Now $2e^\tau U(z-\tau)=\cosh(z)+t/s$ where $|t/s| \leq 1$. This is Hurwitz stable as can be seen by rotating the variables $z \mapsto iz$ (one obtains 
$\cos(z) + t/s$ which is in the Laguerre--P\'olya class \cite{DS}). Similarly,
$$R(z) = (1+r)\cosh\left(\frac{z}{2}\right)+(1-r)\sinh\left(\frac{z}{2}\right) = \sinh\left(\frac{z}{2}+\beta\right) \qquad (\beta\ge 0)$$
is Hurwitz stable ($\beta\ge0$ follows from $\sinh\beta=1+r\ge0$).  Since $H_M$ is Hurwitz stable for each $M=1,2,3,\ldots,$ by Hurwitz's Theorem $F(z)=\lim_{M\to\infty} G_M$ is Hurwitz stable.
\end{proof}
As mentioned earlier, from the integral representation of the Bessel function for $(\Re~\nu+1/2)>0$ \cite[p. 231]{BW},
\begin{equation}\label{bessel}
J_\nu(z) = \frac{2}{\sqrt{\pi}\;\Gamma\left(\nu + \frac{1}{2}\right)}\left(\frac{z}{2}\right)^\nu\int_0^1(1-t^2)^{\nu-1/2} \cos(zt)\; dt,
\end{equation}
$J_0(D)$ will decrease the width of the strip containing the zeros of a real polynomial by the amount given in Theorem \ref{integral-shrink}.  From Theorem \ref{integral-shrink}, it can not be determined if $J_1(D)$ and $J_2(D)$ will also narrow the strip containing the zeros of a real polynomial.
For the case $g(t)=1$ in Theorem \ref{integral-shrink}, a sharp result can be obtained with an identity, suggesting that a tighter bound on the narrowing of the strip width may be obtained for other cases as well.    

\begin{proposition}\label{sinc}
Let $g(z)= \sin(\lambda z )/z$, where $\lambda >0$. Suppose $p(x)\in\SS_\mu(\RR)$. Then 
$
g(D) p(z)\in\SS_\delta(\RR),
$ 
where $\delta=\sqrt{\mu^2 - \lambda^2/3}$.
\end{proposition}

\begin{proof}
Recall the identity 
$$
\frac {\sin(z)} z = \cos(z/2)\cos(z/4)\cos(z/8)\cdots .
$$
Computing the shrinking from $\cos(\lambda D)$ using Theorem \ref{debruijn-jensen} yields that the width of the strip squared decreases as 
$$
\mu^2 \mapsto \mu^2 - \lambda^2 /4 - \lambda^2 /16 -\lambda^2 /64-\cdots =  \mu^2 - \lambda^2 /3. 
$$
\end{proof}

\section{Extensions to Bargmann-Fock spaces}
\label{section:b-f}

In the case that an operator maps polynomials to arbitrary entire functions, then an extension of the characterization of strip preservers can be applied, and the operator can be extended to act on transcendental entire functions, as we obtain below in Theorem \ref{entire-strip}.  Extensions of Theorems \ref{stab-strip} and \ref{integral-shrink} are also stated (Theorems \ref{stab-strip-entire} and \ref{integral-shrink-entire}).
We will work with the class $\SSS_\mu$ (or its image under a linear transformation) which has a characterization in terms of the Hadamard factorization similar to that for $\LP$.

\begin{proposition}[\cite{korevaar}]
A function $\varphi$ is in $\SSS_\mu$ if and only if $\varphi$ has the form 
\begin{equation}\label{stripClosure}
\varphi(z) = c e^{-\gamma z^2 +az}z^m\prod_{k=1}^\nu \left(1-\frac{z}{z_k}\right)e^{z/z_k}, \qquad (0\le\nu\le\infty),
\end{equation}
for $\gamma\ge0$, where $c\in\CC$, $z_k\in\CC\setminus(H_1\cup H_2)$, and
\[
-2\mu\gamma\le \Im~a + \sum_{k=1}^\nu \Im~z_k^{-1} \le 2\mu\gamma.
\]
\end{proposition}

We adopt necessary notation and terminology from \cite{B2}, along with Theorem \ref{symbol-boundop} (below).
For $\alpha\in\NN^n$, define $z^\alpha = z_1^{\alpha_1}z_2^{\alpha_2}\cdots z_n^{\alpha_n}$ and $\alpha !=\alpha_1!\alpha_2!\cdots\alpha_n!$.  Let $\alpha\le\beta$ for $\alpha,\beta\in\RR^n$, denote that $\alpha_j\le\beta_j$ for all $j=1,\ldots,n$, while the condition $\alpha_j>\beta_j$ for all $j=1,\ldots,n$, is denoted $\alpha\gg\beta$.
For $\beta\in(0,\infty)^n$, define the \emph{$\beta$-weighted Bargmann-Fock space}, $\FF_\beta$, as the space of all entire functions $f(z)=\sum_{\alpha\in\NN^n} a_\alpha z^\alpha$, such that
\[
||f||_\beta^2 = \sum_{\alpha\in\NN^n} \frac{\alpha!}{\beta^\alpha} |a_\alpha|^2  <\infty.  
\]

With the definition of the inner product, 
\[
\langle f, g\rangle_\beta = \sum_{\alpha\in\NN^n} \frac{\alpha !}{\beta^\alpha}a_\alpha \overline{b_\alpha} = \frac{\beta_1\cdots\beta_n}{\pi^n}\int_{\CC^n} f(z)\overline{g(z)} \exp\left(-\sum_{i=1}^n\beta_i |z_i|^2 \right)dm,
\]
where $dm$ represents Lebesgue measure, $\FF_\beta$ is a Hilbert space ($||f||^2_\beta=\langle f, f\rangle_\beta$).   If the reproducing kernel is defined by $e_\beta(z,\bar{w}) = \exp(-\sum_{j=1}^n\beta_j z_j \bar{w}_j)$, then
\[
f(w) = \langle f(z), e_\beta(z,\bar{w}) \rangle =  \langle e_\beta(\bar{z},w), f(\bar{z}) \rangle.
\]
By the Cauchy-Schwarz inequality
\begin{equation}\label{bounded-uniform}
|f(w)|^2 = |\langle f(z), e_\beta(z,\bar{w})\rangle|^2 \le ||f||_\beta^2||e_\beta(z,\bar{w})||_\beta^2 = C(|w|)||f||_\beta^2,
\end{equation}
and therefore convergence in $||\cdot||_\beta$ implies uniform convergence on compact subsets of $\CC$ \cite[p. 34]{S}.  

We require the following theorem, which is sharp with respect to the norm parameter (denoted $\gamma$) \cite{B2}.

\begin{theorem}[\cite{B2}]\label{symbol-boundop}
Let $T : \CC[z_1,\ldots, z_n] \rightarrow \CC[[z_1,\ldots, z_m]]$ be a linear operator such that $\overline{G}_T(z,w) \in \FF_{\beta \oplus \gamma}$. Then 
$T$ defines a bounded operator $T : \FF_\alpha \rightarrow \FF_\beta$ for all $\alpha \leq \gamma^{-1}$:
\begin{equation}\label{g-bound}
\|T(f)\|_\beta \leq \|\overline{G}_T(z, \alpha w)\|_{\beta \oplus \alpha} \|f\|_\alpha.
\end{equation}
Moreover $T$ has the integral representation 
\begin{equation}\label{g-int}
T(f)(z)= \int_{\CC^n} f(w)\overline{G}_T(z, -\alpha\overline{w}) \exp\left(-\sum_{i=1}^n\beta_i |w_i|^2 \right)dm.
\end{equation}
Conversely if $T : \FF_\alpha \rightarrow \FF_\beta$ is a bounded operator, then $\overline{G}_T(z,w) \in \FF_{\beta \oplus \gamma}$ for all $\gamma \gg \alpha^{-1}$. 
\end{theorem}

Theorem \ref{entire-strip} extends Theorems \ref{strip} and \ref{trans-strip-char} to transcendental entire functions.

\begin{theorem}\label{entire-strip}
Let $T:\CC[z]\to\CC[[z]]$ be a linear operator of rank greater than one, and let $H_1=\{z\,:\,\Im~z>\mu\}$.  
Then $T:\PP(H_1\cup H^*_1)\to\PPP(H_1\cup H^*_1)$ if and only if $T$ satisfies \textup{(}i\textup{)} and \textup{(}ii\textup{)} below, and consequently $\overline{G}_T(z,-w)\in\FF_{(b,c)}$ for some $b,c\in\RR$. 
\renewcommand\theenumi{\roman{enumi}}
\begin{enumerate}  
\item\label{upper2} $e^{i\mu w}\overline{G}_T(z,w)\in\PPP(H_1\times H)$~~or~~$e^{-i\mu w}\overline{G}_T(z,w)\in\PPP(H_1\times H^*)$; 
\item\label{lower2} $e^{i\mu w}\overline{G}_T(z,w)\in\PPP(H^*_1\times H)$~~or~~$e^{-i\mu w}\overline{G}_T(z,w)\in\PPP(H^*_1\times H^*)$. 
\end{enumerate}
Moreover if $\overline{G}_T(z,-w)\in\FF_{(b,c)}$, then 
\renewcommand\theenumi{\arabic{enumi}}
\begin{enumerate}
\item $T$ extends to a bounded linear operator of form \eqref{g-bound} and \eqref{g-int} for all $a<1/c$, and 
\item $T:\PPP(H_1\cup H^*_1)\cap\FF_a\to\PPP(H_1\cup H^*_1)\cap\FF_b$ for all $a\le\frac{1}{c}$.
\end{enumerate}
\end{theorem}

\begin{proof}
Sufficiency of (i) and (ii) is clear for the same reasons as in the proof of Theorem \ref{strip}, with the transcendental extension of the characterization of stability preservers \cite[Theorem 2.3]{B2}, and the identity $\PPP(H_1)\cap\PPP(H^*_1) = \PPP(H_1\cup H^*_1)$ --- this can be verified from the parameter restrictions of the Hadamard factorizations of $\PPP(H_1)$, $\PPP(H^*_1)$, and $\PPP(H_1\cup H^*_1)$ (equation \eqref{stripClosure}, \cite[Chapter VIII]{Levin}).  As the theorem statement asserts, by \cite[Theorem 6.6]{BB2}, if $e^{i\mu w}\overline{G}_T(z,w)\in\PPP(H_1\times H)\cup\PPP(H_1\times H^*)$, then $\overline{G}_T(z,-w)\in\FF_{(b,c)}$ for some $b,c\in\RR$.  The extension described by (1) and (2) then follows from Theorem \ref{symbol-boundop} and the following limiting argument.  For $f\in\PPP(H_1\cup H^*_1)\cap\FF_a$, let $\{f_n\}_{n=0}^\infty$ be a sequence of polynomials such that $f_n\to f$ locally uniformly, then $T(f_n)\to T(f)$ locally uniformly by \eqref{bounded-uniform} and \eqref{g-bound}.  By Hurwitz's theorem, $T(f)\in\PPP(H_1\cup H^*_1)$.

To prove necessity of conditions (i) and (ii), we define two linear operators by their action on shifted monomials:
\begin{align*}
J^{-\mu}_n((z-i\mu)^k) &= \frac{k!}{n^k}\binom{n}{k}(z-i\mu)^k, \text{and}\\
J^{+\mu}_n((z+i\mu)^k) &= \frac{k!}{n^k}\binom{n}{k}(z+i\mu)^k, \text{for each } k\in\NN.
\end{align*}
By the criteria for stability preservers \cite{BB1}, $J^{-\mu}_n:\PPP(H_1)\to\PP_n(H_1)$ and $J^{+\mu}_n:\PPP(H^*_1)\to\PP_n(H^*_1)$, and furthermore $J^{\pm\mu}_n\circ T(p)\to T(p)$ locally uniformly as $n\to\infty$, for all $p\in\CC[z]$ \cite[Lemma 2.2]{CC}.  Therefore the following claim implies conditions (i) and (ii) through the transcendental characterization of stability preservers \cite[Theorem 6]{BB1}.\\
\\
{\bf Claim:} For each $n\in\NN$, $J^{-\mu}_n\circ T:\PP(H_1)\to\PP_n(H_1)\cup\{0\}$ or $J^{-\mu}_n\circ T:\PP(H^*_1)\to\PP_n(H_1)\cup\{0\}$, and similarly  $J^{+\mu}_n\circ T:\PP(H_1)\to\PP_n(H^*_1)\cup\{0\}$ or $J^{+\mu}_n\circ T:\PP(H^*_1)\to\PP_n(H^*_1)\cup\{0\}$.\\

To prove the claim, let $m\in\NN$, and note that Lemma \ref{line} implies either  $J^{-\mu}_n\circ T:\PP_m(H_1)\to\PP_n(H_1)\cup\{0\}$ or $J^{-\mu}_n\circ T:\PP_m(H^r_1)\to\PP_n(H_1)\cup\{0\}$ (where $U^r=\Int(C\setminus U$)).  If $J^{-\mu}_n\circ T:\PP_m(H^r_1)\to\PP_n(H_1)\cup\{0\}$, one can proceed with an argument similar to that used in the proof of Theorem \ref{strip} to show that $J^{-\mu}_n\circ T:\PP_m(H^*_1)\to\PP_n(H_1)\cup\{0\}$.  Repeating this argument for $J^{+\mu}_n\circ T$ establishes that 
\renewcommand\theenumi{\Alph{enumi}}
\begin{enumerate}
\item $J^{-\mu}_n\circ T:\PP_m(H_1)\to\PP_n(H_1)\cup\{0\}$ or $J^{-\mu}_n\circ T:\PP_m(H^*_1)\to\PP_n(H_1)\cup\{0\}$, and 
\item $J^{+\mu}_n\circ T:\PP_m(H_1)\to\PP_n(H^*_1)\cup\{0\}$ or $J^{+\mu}_n\circ T:\PP_m(H^*_1)\to\PP_n(H^*_1)\cup\{0\}$.
\end{enumerate}
Observe that one of the conditions in (A) must hold for all $m\in\NN$ and one of the conditions in (B) must hold for all $m$. Indeed, $J^{-\mu}_n\circ T:\PP_{m_1}(H_1)\not\to\PP_n(H_1)\cup\{0\}$ and $J^{-\mu}_n\circ T:\PP_{m_2}(H_1)\to\PP_n(H_1)\cup\{0\}$ for $m_2>m_1$ is a contradiction, as $\PP_{m_2}(H_1)\supseteq\PP_{m_1}(H_1)$. This establishes the claim.
\end{proof}

Since both of Theorems \ref{stab-strip} and \ref{integral-shrink} involve only differential operators with constant coefficients, we extend them with a more convenient statement, which may be proved using Theorem \ref{symbol-boundop} (see \cite[Example 3.6]{B1}).  

\begin{theorem}[{\cite[Chapter IX, Theorem 8]{Levin}}]\label{inf-diff-op}
Let $f(z),\varphi(z)\in\PPP(H)$, where $f(z)=e^{-\gamma_1 z^2}f_1$, $\varphi=e^{-\gamma_2 z^2}\varphi_1$, and $f_1$ and $\varphi_1$ are entire functions of genus one.  If $\gamma_1\gamma_2<1/4$, then $f(D)\varphi(z)\in\PPP(H)$ and the series 
\begin{equation}\label{diffop-series}
f(D)\varphi(z)= a_0\varphi(z)+a_1\varphi'(z)+a_2\varphi''(z)+\dots+ a_k\varphi^{(k)}(z)+\cdots
\end{equation} 
converges uniformly on any compact subset of $\CC$.  \textup{(}Furthermore, $f(D):\FF_a\to\FF_b$ for all $a<1/c$, $b>c/(1-ac)$ \cite[Proposition 2.5]{LS}\textup{)}.
\end{theorem}
 
The uniform convergence in \eqref{diffop-series} is proved by showing (and also implies) uniform boundedness of the partial sums, thus if $p_n$ is any sequence of approximating polynomials which converges uniformly to $\varphi$ on compact subsets of $\CC$, $f(D)p_n\to f(D)\varphi$ uniformly as well.  From standard limiting arguments the extensions of Theorems \ref{stab-strip} and \ref{integral-shrink} follow with the aid of Theorem \ref{inf-diff-op}.  A finite Fourier transform has order at most one, and whence there is no need to restrict $f(D)$ to act on a subclass of $\SSS_\mu(\RR)$ in Theorem \ref{integral-shrink-entire}.

\begin{theorem}\label{stab-strip-entire}
Suppose $T=\sum_{k=0}^\infty a_k D^k = h(D)e^{i\lambda D}$ is a differential operator with constant coefficients $a_k\in\CC$, $k=0,1,2\ldots,$ and that $h\in\PPP(H^*)$.  For  $\mu\ge\lambda\ge 0$, $a<1/(4\gamma)$, $b>a/(1-4\gamma a)$,
$(T+T^*):\FF_a\to\FF_b$ and
\[
(T+T^*):\FF_a\cap\SSS_\mu(\RR)\to\SSS_{\delta}(\RR),
\]
where $\delta=\max\left\{\sqrt{\mu^2-\lambda^2},0\right\}$. 
\end{theorem}

\begin{remark}
If $\overline{G}_T(z,w)\in\FF_\beta$, then $\overline{G}_{(T+T^*)}(z,w)\in \FF_\beta$.  In Theorem \ref{stab-strip-entire}, $T+T^*=e^{-\gamma D^2}g(D)$, where $g$ is an entire function of order at most $1$. 
\end{remark}

\begin{theorem}\label{integral-shrink-entire}
Let $f(z)$ be an entire function with the representation 
\[
f(z) = \int_0^1 \cos(\lambda zt) g(t) \;dt, 
\]
where the function $g$ is non-negative and non-decreasing, and $\lambda\in\RR$.  Then $f(D):\SSS_\mu(\RR)\to\SSS_\delta(\RR)$, where $\delta=\max\left\{\sqrt{\mu^2-\lambda^2/4},0\right\}$. 
\end{theorem}

We end by stating a necessary and sufficient condition for an entire function to have zeros only in a strip.  This inequality reduces to one already established for the real line in the case the strip width is set to zero \cite{CV}.

\begin{proposition}\label{ineq:impart}
Let $f$ be a real entire function.  Then $f\in\SSS_\mu(\RR)$ if and only if $f$ has order at most two, and either for all $\Im~z>\mu$ 
\[
\Im\{-f'(z)\overline{f(z)}\}> 0,
\]
 or $\Im\{-f'(z)\overline{f(z)}\}\equiv 0$.
\end{proposition}

\begin{proof}
Since any function $f\in\SSS_\mu(\RR)$ has order at most two (see \eqref{stripClosure}), and this is a hypothesis on $f$ in the second clause, let us continue assuming that $f$ has order two or less. 
By Theorem \ref{constant-coefficients}, if $f\in\PPP_\mu$, then $(1+iD)f = f+if'$ is also in $\PPP_\mu$. If f is real, then $f\in\SSS_\mu(\RR)$ if and only if  $(1+iD)f=f+if'\in\PPP_\mu$ by Theorems \ref{constant-coefficients}, \ref{inf-diff-op}, and Proposition \ref{h-star-bar-conditions}. If $f+if'\in\PPP_\mu$, then either $f+if'\in\PPP_\mu\cap\HD$ or $f+if'\in\PPP_\mu\cap\HDD\setminus\HD$. By Lemma \ref{wronskians}, $f+if'\in\PPP_\mu\cap\HD$ if and only if
$W_\lambda[f',f](x)<0$ for all $\lambda>\mu$, $x\in\RR$, or equivalently 
\[
2 W_\lambda[f,f'] = \frac{1}{\lambda} \Im\{-f'(z)\overline{f(z)}\}> 0,
\]
where $z=x+i\lambda$.
Otherwise, by Lemma \ref{h-star-bar-conditions} (iii), $f+if'\in\PPP_\mu\cap\HDD\setminus\HD$ if and only if 
\begin{equation}\label{eq:Dboundary}
(a+ib)(f + if')=\varphi,
\end{equation}
where $\varphi\in\SSS_\mu(\RR)$, $a,b\in\RR$. The imaginary part of \eqref{eq:Dboundary} is $bf+af'=0$, which implies that $f$ (and also $\varphi$) is of the form $f=ce^{rz}$ for $c\in\CC$ and $r\in\RR$, whence $\Im\{-f'(z)\overline{f(z)}\}\equiv 0$. 
\end{proof}

\section{Fourier transforms with only real zeros}
\label{fourier-kernel}
N.~G.~de~Bruijn and L.~Ilieff independently proved that if $\phi(t)$ is an even entire function and $\phi'(t)\in\LP$, then the Fourier transform of $e^{\phi(it)}$ will have only real zeros, provided it exists \cite{D,I}.  We find a new proof for a result of de~Bruijn (Theorem \ref{fourier-polynomial-even}) by modifying Ilieff's method of proof, and extend it in Corollary \ref{fourier-polynomial}, and Proposition \ref{fourier-cor}.  Similar results have been obtained by H. Ki and Y.-O. Kim \cite{KKFT} by extending de~Bruijn's argument involving the saddle point method.  

\begin{lemma}
If $G\in\CC[z]$ with $\deg(G)=2d\neq 0$, $G(it)=-ct^{2d}+\cdots$, and $c>0$, then as $m\to\infty$, the roots of $G(it)=-m$, denoted $\{t_k\}_{k=1}^{2d}$, satisfy
\begin{equation}\label{unity-roots}
t_k=e^{i2\pi k/2d}\left(\frac{m}{c}\right)^{1/2d} + O\left(1\right) \qquad (m\to\infty),
\end{equation}
for $k= 1, 2, 3, \ldots, 2d$.  
\end{lemma}

\begin{proof}
Let $m=c\lambda^{2d}$, $t=\lambda x$, and substitute these in the equation $G(it)=m$, yielding
\begin{equation}\label{uroots2}
x^{2d} + \frac{1}{\lambda} O(x^{2d-1})=1.
\end{equation} 
Thus, by Hurwtiz's theorem the roots of \eqref{uroots2} approach the roots of unity, and from the local analyticity of the roots of \eqref{uroots2} in $s=1/\lambda$ around $s=0$ (see \cite{Brillinger}), the roots are given by
\[
x=e^{i2\pi k/2d}+O(1/\lambda) \qquad (\lambda\to\infty),
\]
$k= 1, 2, 3, \ldots, 2d$.  Then, replacing $x=t/\lambda$ and $\lambda=(m/c)^{1/2d}$ yields \eqref{unity-roots}. 
\end{proof}

\begin{lemma}\label{jensen}
There is a real constant $A>1$ such that
\begin{equation}\label{exp-jensen-estimate}
\left|1+\frac{z}{n}\right|^n\le e^{\Re\;z/2},
\end{equation}
for all $z\in\CC$ with $-nA < \Re~z \le 0$ and $|\Im~z/\Re~z|<1/2$.
\end{lemma}

\begin{proof}
With $z=x+iy$, the inequality in \eqref{exp-jensen-estimate} can be rewritten as 
\begin{equation}\label{FXeps}
e^X - (1+X)^2 - Y^2 \ge 0,
\end{equation}
where $X=x/n$, and $Y=y/n$. 
For \eqref{FXeps} to hold when $|y/x| \le 1/2$, it is sufficient that 
\begin{equation}\label{FXeps2}
F(X) = e^X - (1+X)^2 - X^2/4 \ge 0.
\end{equation}
It follows from $F(0) = 0$, and
\[\frac{dF}{dX}(0) = -1,\]
that there is always an interval $-A<X\le0$ where \eqref{FXeps2} holds. The critical points of $F$ occur at the intersection points of the convex function $e^X$ and the line $5X/2 + 2$.  Since $F(-1)=1/e-1/4>0$, and $F$ has only one critical point for $X<0$, it follows that $-A<-1$.
\end{proof}

\begin{lemma}\label{Fnm}
Let $G$ be a real polynomial of the form $G(it)=-ct^{2d}+\cdots$, where $\deg(G)=2d\neq 0$, and $c>0$, such that $G'\in\LP$.  For $m,n\in\NN$, let
\begin{equation}\label{fm-approx-integral}
F_{n,m}(z) = \int_{-\bar{a}_m}^{a_m} (1+G(it)/m)^ne^{itz}~dt,
\end{equation}
where $G(ia_m)=-m=\overline{G(ia_m)}=G(-i\bar{a}_m)$.  Then $F_{m,n}\in\LP$ for all $n,m\in\NN$.
\end{lemma}

\begin{proof}
Let $n\ge 1$.  Integration of $F_{n,m}$ by parts yields 
\begin{align*}
\int_{-\bar{a}_m}^{a_m} (1+G(it)/m)^ne^{itz}~dt &= \left. \frac{(1+G(it)/m)^ne^{itz}}{iz}\right|_{-\bar{a}_m}^{a_m}\\
&\qquad - \int_{-\bar{a}_m}^{a_m} \frac{n}{z}G'(it)(1+G(it)/m)^{n-1}e^{itz}~dt\\
 &= \frac{n}{z} G'(D)\int_{-\bar{a}_m}^{a_m} (1+G(it)/m)^{n-1}e^{itz}~dt.
\end{align*}
Since $G'\in\LP$, the operator $G'(D)$ preserves reality of zeros by the Hermite-Poulain Theorem.  It follows by induction on $n$ that $F_{n,m}\in\LP$ for all $m,n\in\NN$, once it is established that the base case has only real zeros for any $m\in\NN$.  Indeed,
\[
F_{0,m} = \int_{-\bar{a}_m}^{a_m} e^{itz}~dt = \frac{2 e^{-a_y z}\sin(a_x z)}{z}\in\LP,
\]
for any $a_m=a_x+i a_y\in\CC$ ($a_x, a_y\in\RR$).
\end{proof}

\begin{theorem}[{\cite[Theorem 20]{D}}]\label{fourier-polynomial-even}
Let $F(z)$ be an entire function possessing the representation
\begin{equation}\label{f-trans}
F(z) = \int_{-\infty}^\infty e^{G(it)}e^{itz}~dt,
\end{equation}
where $G'(x)\in\RR[z]\cap\LP$, and $G(it)=-ct^{2d}+\cdots$, where $\deg(G)=2d\neq 0$, and $c>0$. Then $F(z)\in\LP$.    
\end{theorem}

\begin{proof}
Let $G$ be a real polynomial as described in the theorem, suppose that $c=1$, and $G(ia_m)=-m=\overline{G(ia_m)}=G(-i\bar{a}_m)$, and $a_m=t_{2d}$ as described in \eqref{unity-roots}. Expanding around $t=a_m$, 
\begin{align*}
G(it)+m &= 0 + iG'(ia_m)(t-a_m) + \cdots \\
& \qquad \cdots + i^{2d-1}\frac{G^{(2d-1)}(ia_m)}{(2d-1)!}(t-a_m)^{2d-1} - (t-a_m)^{2d}
\end{align*}

Let $a_m=x_m+iy_m$, where $x_m,y_m\in\RR$.  By \eqref{unity-roots}, $|t-a_m|=O(1)$ as $m\to\infty$ for values of $t$ in the line segment $L_m=\{x_m+is : 0\le s\le y_m\}$, and using $G^{(k)}(ia_m)=O((a_m)^{2d-k})=O(m^{(2d-k)/2d})$ one obtains
\begin{equation}\label{G-estimate}
G(it)+m = O(m^{\frac{2d-1}{2d}}),\qquad (m\to\infty),
\end{equation}
for $t\in L_m$, and similarly for $t\in-L^*_m$. 
Then, 
\[
F_{n,m}(z) = \int_{-L^*_m} I_{n,m}(t,z) dt +  \int_{-x_m}^{x_m} I_{n,m}(t,z) dt + \int_{L_m} I_{n,m}(t,z) dt,
\]
where $I_{n,m}(t,z) = (1+G(it)/m)^ne^{itz}$.

Choosing $n=m$ and using the estimates \eqref{G-estimate} and \eqref{unity-roots} yields 
\[
(1+G(it)/m)^m = (Cm^{-\frac{1}{2d}}+O(m^{-\frac{2}{2d}}))^m = O(C^mm^{-\frac{m}{2d}})
\]
 for $t\in L_m\cup(-L^*_m)$, some $C\in\CC$.  Therefore,
\begin{equation} \label{Fmm-estimate}
F_{m,m}(z) = \int_{-x_m}^{x_m} (1+G(it)/m)^me^{itz}dt + O(C^mm^{-\frac{m}{2d}}), \qquad (m\to\infty).
\end{equation}
Let $A$ be the constant in Lemma \ref{jensen}.  Since the interval where \eqref{exp-jensen-estimate} holds grows more quickly with $n=m$ than $|G(ix_m)|$ by \eqref{G-estimate}, there is an $M\in\NN$ such that that $Am> \Re~G(ix_m)$ for all $m\ge M$.
By Lemma \ref{exp-jensen-estimate}, 
\begin{equation}
\left|\int_{-x_m}^{-R} (1+G(it)/m)^me^{itz}dt\right| \le \int_{-\infty}^{-R} e^{\Re\;G(it)}e^{-t\:\Im z}dt, \text{ and} \label{tail-bound1}\\
\end{equation}
for real $R$ sufficiently large that $|\Im\;G(it)/\Re\;G(it)|<1/2$ for all $|t|>R$, and for $m>M$ such that $x_m>R$.  Inequality \eqref{tail-bound1} holds {\em mutatis mutandis} when the limits of integration on the left-hand side are changed to $R$ and $x_m$, and on the right to $R$ and $\infty$, repectively.  The integral on the right-hand side of \eqref{tail-bound1} converges and provides a uniform bound over compact subsets of $\CC$ for the integral on the left hand side (likewise for the integral on $[R,x_m]$).  Furthermore, by choosing $R$ sufficiently large for fixed $z\in\CC$, the integral on the right hand side of \eqref{tail-bound1} (and also for the integral on $[R,\infty]$) can be made arbitrarily small, hence $F_{m,m}\to F$ pointwise.  Since $(1+G(it)/m)^m$ converges uniformly on $[-R,R]$ as $m\to\infty$, 
\[
\int_{-R}^{R} (1+G(it)/m)^me^{itz}dt
\]
is locally uniformly bounded on $\CC$.  With \eqref{Fmm-estimate} and \eqref{tail-bound1}, we may conclude that the sequence $F_{m,m}$ is locally uniformly bounded and therefore $F_{m,m}\to F$ uniformly on compact sets as $m\to\infty$ \cite[p. 34]{S}. 
  Since each $F_{m,m}\in\LP$ by Lemma \ref{Fnm}, by Hurwitz's Theorem $F(z)\in\LP$. 
\end{proof}

\begin{corollary}\label{fourier-polynomial}
Let $F(z)$ be an entire function possessing the representation \eqref{f-trans}
where $\Re~G(it)\to -\infty$ as $t\to \pm\infty$, and $G'(x)\in\RR[z]\cap\LP$. Then $F(z)\in\LP$.
\end{corollary}

\begin{proof}
Given any $G(t)\in\RR[t]$ with $\Re~G(it)\to -\infty$ as $t\to \pm\infty$ the integral \eqref{f-trans} converges.  If the leading term of $G$ has even degree the conclusion holds by \ref{fourier-polynomial-even}.  Suppose the leading term of $G$ has odd degree $2k-1$, $k\in\NN$.  We can perturb the coefficients of $G$ such that $G'$ has only simple real zeros, separated by some distance $\varepsilon>0$; call this perturbation $G_\varepsilon$.   The convergence of the integral in \eqref{f-trans} and the fact that $G_\varepsilon$ is a polynomial, imply that there exist $B,C>0$ such that $|e^{G_\varepsilon(it)}|<|Ce^{-Bt^2}|$ for all $t\in\RR$, and $0\le \varepsilon \le\varepsilon_0$, provided $\varepsilon_0$ is sufficiently small.  Then for $\delta,\varepsilon>0$, with $\delta$ sufficiently small and $\varepsilon\le\varepsilon_0$, 
\[
F_{\varepsilon,\delta}(z) = \int_{-\infty}^\infty e^{-\delta t^{2k} + G_\varepsilon(it)}e^{itz}~dt \in \LP 
\] 
by Theorem \ref{fourier-polynomial-even}.
Local uniform boundedness, given by
$$|F_{\varepsilon, \delta}(z)|\le  \int_{-\infty}^\infty |e^{G_\varepsilon(it)}|e^{-t\; \Im~z}~dt \le \int_{-\infty}^\infty |Ce^{-Bt^2}|e^{-t\; \Im~z}~dt,$$
 and pointwise convergence to $F_{\varepsilon,0}$ as $\delta\to 0$, imply that $F_{\varepsilon,\delta}(z) \to F_{\varepsilon,0}(z)$ locally uniformly for $0<\varepsilon\le\varepsilon_0$.  By Hurwitz's theorem, $F_{\varepsilon,0}(z)\in\LP$ for $0<\varepsilon\le\varepsilon_0$.  Letting $\varepsilon\to 0$ and applying Hurwitz's theorem again yields $F_{0,0}(z)=F(z)\in\LP$.
\end{proof}

We now extend a Theroem of Ilieff and de~Bruijn, who proved that for any even entire function $G$ with in $G'\in\LP$, the Fourier transform $F$ in \eqref{f-trans} has only real zeros.  Proposition \ref{fourier-cor} below removes the requirement to be even, and requires only that the coefficients in the real part of $G(it)$ are eventually negative.

\begin{proposition}\label{fourier-cor}
Suppose $G(t)=\sum_{k=0}^\infty \gamma_k x^k/k!$ is a real transcendental entire function, $G'\in\LP$, and there exists a $K\in\NN$ such that $(-1)^K\gamma_{2K}< 0$ and $(-1)^k\gamma_{2k}\le 0$ for $k>K$.  Then
\[
F(z)=\int_{-\infty}^\infty e^{G(it)}e^{izt}dt \in \LP.
\]
\end{proposition}

\begin{proof}
Let $g_n(t)=\sum_{k=0}^n\binom{n}{k}\gamma_k (t/n)^k$ and define
\[
M(t) = \sum_{k=0}^{K-1} \frac{|\gamma_{2k}|}{(2k)!}t^{2k} + (-1)^K\gamma_{2K}\left(\frac{t}{2K}\right)^{2K}.
\]
If $G'(t)\in\LP$, then $g'_n(t)=\sum_{k=0}^{n-1}\binom{n-1}{k}\gamma_{k+1} (t/n)^k\in\LP$ for each $n\in\NN$ \cite{CC}.
By virtue of the inequality $\binom{n}{k}/n^k \le 1/k!$, 
\[
|e^{g_n(it)}|\le e^{M(t)} \text{ for all } t\in\RR, n\in\NN,
\]
and since the leading coefficient of $M(t)$ is negative and of degree at least two, $e^{M(t)} \le e^{-\varepsilon t^2}$ for a fixed $\varepsilon>0$ and $|t|>R_\varepsilon$ sufficiently large.  Thus,
\[
\left|\int_{-\infty}^\infty e^{g_n(it)}e^{izt} dt\right| < \int_{-R_\varepsilon}^{R_\varepsilon} e^{M(t)}e^{-t\:\Im z} dt + \int_{-\infty}^{-R_\varepsilon} e^{-\varepsilon t^2} dt + \int_{R_\varepsilon}^\infty  e^{-\varepsilon t^2} dt,
\]
and the left hand side is uniformly bounded on compact subsets of $\CC$.  Thus,
\[
\int_{-\infty}^\infty e^{g_n(it)}e^{izt} dt \to F(z)
\]
locally uniformly on $\CC$ as $n\to\infty$ \cite[p. 34]{S}, whence $F(z)\in\LP$ by Theorem \ref{fourier-polynomial} and Hurwitz's Theorem.
\end{proof}

\begin{remark}
Let $\phi(t)=\sum_{k=0}^\infty a_k t^k$ be a real entire function.  Then for any real polynomial $p(t)=\sum_{k=0}^N b_kt^k$, the $k$-th coefficient of $p(t)\phi(t)=\sum_{k=0}^\infty c_k t^k$, is 
\[
c_k = b_0a_k + b_1a_{k-1} + \cdots + b_Na_{k-N} \qquad (k\ge N).
\]  
Since $\phi$ is entire, $\limsup_{k\to\infty} |a_{k+1}/a_k|= 0$, and consequently for $k$ sufficiently large, ${\rm sgn}(c_k) = {\rm sgn}(b_Na_{k-N})$.
Thus for $p\in\RR[z]$, such that $\Re~p(it)$, $t\in\RR$, has positive leading coefficient, and $G$ as described in Proposition \ref{fourier-cor} with $G'\in\LP$ replaced by $(pG)'\in\LP$, the function $G_2(t)=p(t)G(t)$ again satisfies the hypotheses for $G$ in Proposition \ref{fourier-cor}.
\end{remark}

\section{Closing remarks}

Below we identify several avenues one might explore as a continuation of the work here.  As mentioned earlier, D.~Cardon has shown that an operator with property \eqref{su1} of strong universal factors must have order at least one and either mean or maximal type \cite{cardon}.  With this in mind, we list some directions for further investigation.
\begin{enumerate}\itemsep 2pt
\item
Obtain a characterization of real strip preservers which either uses properties of the operator symbol, or some other readily verifiable criteria. 
\item
Determine if operators arising from Theorem \ref{stab-strip} which possess property \eqref{su1} of a strong universal factors also satisfy property \eqref{su2}.
\item 
Given $f\in\LP$ of order one and mean type, determine a partner $g\in\LP$ such that $(f+ig)(D)$ has the form of $T$ in Theorem \ref{stab-strip} (note the zeros of $f$ and $g$ must interlace).  Do all operators with order strictly less than two, and property \eqref{su1} of strong universal factors, have this form?
\item
Prove the following conjecture,
\begin{conjecture}
Suppose $G(t)=\sum_{k=0}^\infty \gamma_k x^k/k!$ is a real transcendental entire function, $G'\in\LP$, and $G(it)\to -\infty$ as $t\to\pm\infty$.  Then
\[
F(z)=\int_{-\infty}^\infty e^{G(it)}e^{izt}dt \in \LP.
\]
\end{conjecture}

\end{enumerate}

\newpage

\appendix
\section{Linear preservers and strip regions under M\"obius transformations} \label{appendix1}
Recall the general form of a M\"obius transformation, given by 
\begin{equation}\label{mobius}
\Phi(z) = \frac{az+b}{cz+d},\qquad a,b,c,d\in\CC,\qquad ad-bc\neq 0,
\end{equation}
with inverse $\Phi^{-1}(z) = (dz-b)/(-cz+a)$.
The images of a strip under M\"obius transformations, which are not of the form $z\to az+b$, take three basic geometries, which are depicted below.  
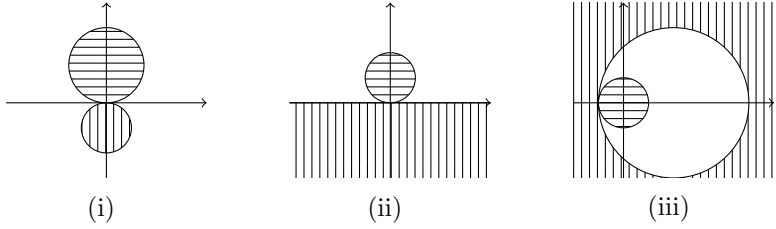
\begin{center}
\parbox{1 in}{
\begin{tikzpicture}[scale=1/3]
\draw[pattern=horizontal lines, pattern color=black] (0,1.5) circle (1.5);
\draw[pattern=vertical lines, pattern color=black] (0,-1) circle (1);
\draw[->] (-4,0)--(4,0);
\draw[->] (0,-3)--(0,4); 
\end{tikzpicture}
\begin{center}(i)\end{center}
}
\hspace{1 cm}
\parbox{1 in}{
\begin{tikzpicture}[scale=1/3]
\draw[pattern=horizontal lines, pattern color=black] (0,1) circle (1);
\draw[->](0,-3)--(0,4);
\draw[->](-4,0)--(4,0);
\clip (-4,-3) rectangle (4,0);
\draw[pattern=vertical lines, pattern color=black] (-4.1,-3.1) rectangle (4.1,0);
\end{tikzpicture}
\begin{center}(ii)\end{center}
}
\hspace{1 cm}
\parbox{1 in}{
\begin{tikzpicture}[scale=1/3]
\clip (-2,-3) rectangle (6,4);
\draw[pattern=vertical lines, pattern color=black] (-2.1,-3.1) rectangle (6.1,4.1);
\fill[white] (2,0) circle (3);
\draw (2,0) circle (3);
\draw[pattern=horizontal lines, pattern color=black] (0,0) circle (1);
\draw[->](0,-3)--(0,4);
\draw[->](-2,0)--(6,0);
\end{tikzpicture}
\begin{center}(iii)\end{center}
}
\end{center}

Let the horizontal lines and the vertical lines mark the images of the open half-plane above the strip $S_{\mu=1}$, and the open half-plane below $S_{\mu=1}$ (resp.), then the examples of images above are given by
\begin{enumerate}\renewcommand\theenumi{\roman{enumi}}
\item
$\Phi(z)=-1/z$ (the exterior of two tangent disks), 

\item
$\Phi(z)=-1/(z+i)$ (the exterior of a disk tangent to a plane),

\item 
$\Phi(z)=(4-iz)/(2+iz)$ (a disk with a tangent interior disk removed).

\end{enumerate}

The set of polynomials of exact degree $n$ whose zero set is some bounded region $\Omega\subset\CC$ is not dense in $\PP_n(\Omega^c)$ (see \cite[p. 473]{BB1}).  Thus a characterization for operators preserving stability on regions which include case (iii) can only be algebraic.  For the proofs of Corollary \ref{cor-mobius} and \ref{cor-mobius-n} below we require the following lemma.

\begin{lemma}[{\cite[Lemma 6]{BB1}, modified}]\label{lem-mobius-image}
Let $T:\CC_n[z]\to\CC_m[z]$ and let $m$ be minimal, i.e., $m=\max\{\deg(T(p)) : p\in\CC_n[z]\}$.  Let $C_j=\Phi^{-1}(H_j)$, for $j=1,2$, be open circular domains, where $\Phi$ is a M\"obius transformation as in \eqref{mobius} and $H_1$ and $H_2$ are parallel open half-planes as in Theorem \ref{strip}.  Let $S$ be the linear operator defined by $\phi^{-1}_m T\phi_n$, where $\phi_k(p)=(cz+d)^kp(\Phi(z))$.  The following are equivalent:
\renewcommand\theenumi{\roman{enumi}}
\begin{enumerate}
\item $T(p)$ is $C_1\cup C_2$-stable or zero whenever $p$ is of degree $n$ and $C_1\cup C_2$-stable.
\item $S(p)$ is $H_1\cup H_2$-stable or zero whenever $p$ is of degree $n$ and $H_1\cup H_2$-stable.
\item $S(p)$ is $H_1\cup H_2$-stable or zero whenever $p$ is of degree at most $n$ and $H_1\cup H_2$-stable.
\end{enumerate}
\end{lemma}

The proof of Lemma \ref{lem-mobius-image} is nearly identical to that for \cite[Lemma 6]{BB1} and is therefore omitted.  The only scenario where the equivalence is not straight forward is when either $C_1$ or $C_2$ is the exterior of a circle, and thus $S(p)=(-cz+a)^{r(p)}S_0(p)$, where $S_0(p)$ is $H_1\cup H_2$-stable and $a/c\in H_1\cup H_2$.  In this case, the minimality of $m$ and a continuity argument (as in the proof of \cite[Lemma 6]{BB1}) establish that $r(p)=0$ for all $p\in\PP(H_1\cup H_2)$.

Lemma \ref{lem-mobius-image} with Theorem \ref{strip} yields a condition for preservers of polynomials of fixed degree $n$ with zeros which lie in a M\"obius image of the strip.

\begin{corollary}\label{cor-mobius}
Let $n\in\NN$, $T:\CC_n[z]\to\CC[z]$ be a linear operator.  Let $\Phi(z)$ be a M\"obius transformation, and $C_1,C_2\subset \CC$ open circular domains such that $\Phi^{-1}(\CC\setminus(C_1\cup C_2))$ is a closed strip in the complex plane.  Then $T:\PP_n(C_1\cup C_2)\setminus\PP_{n-1}(C_1\cup C_2)\to\PP(C_1\cup C_2)\cup\{0\}$ if and only if either 
\renewcommand\theenumi{\roman{enumi}}
\begin{enumerate}
\item
$T$ has range of dimension at most one and is of the form
\[
T[f]=\alpha(f)p \text{ for } f\in\CC_n[z]\setminus \CC_{n-1}[z],
\]
where $\alpha:\CC_n[z]\setminus \CC_{n-1}[z]\to\CC$ is a linear functional and $p\in\PP(C_1\cup C_2)$; or
\item
$T$ satisfies \textup{(}a\textup{)} and \textup{(}b\textup{)} below.
\begin{enumerate}  
\item $T : \PP_n(C_1)\setminus\PP_{n-1}(C_1) \rightarrow \PP(C_1)\cup\{0\}$\\
 or $T : \PP_n(C_2)\setminus\PP_{n-1}(C_2) \rightarrow \PP(C_1)\cup\{0\}$;
\item $T : \PP_n(C_1)\setminus\PP_{n-1}(C_1) \rightarrow \PP(C_2)\cup\{0\}$ \\
or $T : \PP_n(C_2)\setminus\PP_{n-1}(C_2) \rightarrow \PP(C_2)\cup\{0\}$.
\end{enumerate}
\end{enumerate}
Additionally, $T:\PP_n(\overline{C_1}\cup\overline{C_2})\setminus\PP_{n-1}(\overline{C_1}\cup \overline{C_2})\to\PP(\overline{C_1}\cup \overline{C_2})$ if and only if $T:\PP_n(C_1\cup C_2)\setminus\PP_{n-1}(C_1\cup C_2)\to\PP(C_1\cup C_2)$ and $T[z^k]\in\PP(\overline{C_1}\cup \overline{C_2})$, where $k = \min\{ m :T[z^m]\not\equiv 0\}.$
\end{corollary}

\begin{proof}
Sufficiency of the conditions (a) and (b) is immediate, as is necessity of (a) when $T$ has a rank at most one.  To show necessity, we may equivalently show that $S$, as defined in Lemma \ref{lem-mobius-image}, preserves $H_1\cup H_2$-stability for polynomials of degree $n$ or less.  By Theorem \ref{strip}, the conditions necessary for $S$ to preserve $H_1\cup H_2$-stability correspond to those given when Lemma \ref{lem-mobius-image} part (ii) is applied to obtain the equivalent conditions for $T$.  The conditions to preserve $\overline{C_1}\cup\overline{C_2}$-stability for polynomials of precise degree $n$ now follow from Theorem \ref{open-set}.
\end{proof}

In Corollary \ref{cor-mobius}, let $\Phi$ have the form \eqref{mobius} with $\Phi^{-1}(\CC\setminus(C_1\cup C_2))=S_\mu$, and without loss of generality assume $\Phi^{-1}(C_1)$ is the open half-plane above $S_\mu$. One may then obtain the following explicit algebraic versions of the conditions in Corollary \ref{cor-mobius} for operators of rank one or more. 

\begin{enumerate}  
\item $T[((az+b)+(cz+d)(w-i\mu))^n]\in\PP(C_1\times H)$\\
or $T[((az+b)+(cz+d)(w+i\mu)^n]\in\PP(C_1\times H^*)$; 
\item  $T[((az+b)+(cz+d)(w-i\mu))^n]\in\PP(C_2\times H)$\\
 or $T[((az+b)+(cz+d)(w+i\mu)^n]\in\PP(C_2\times H^*)$. 
\end{enumerate}

\noindent
The algebraic characterization of preservers of polynomials with zeros in M\"obius images of the strip now follows from Corollary \ref{cor-mobius}, which is used to verify the conditions (a) and (b) in case (ii).
\begin{corollary}\label{cor-mobius-n}
Let $n\in\NN$, $T:\CC_n[z]\to\CC[z]$ be a linear operator.  Let $\Phi(z)$ be a M\"obius transformation, and $C_1,C_2\subset \CC$ open circular domains such that $\Phi^{-1}(\CC\setminus(C_1\cup C_2))$ is a closed strip in the complex plane.  Then $T:\PP_n(C_1\cup C_2)\to\PP(C_1\cup C_2)\cup\{0\}$ if and only if either
\renewcommand\theenumi{\roman{enumi}}
\begin{enumerate}
\item
$T$ has range of dimension at most one and is of the form
\[
T[f]=\alpha(f)p \text{ for } f\in\CC_n[z]
\]
where $\alpha:\CC_n[z]\to\CC$ is a linear functional and $p\in\PP(C_1\cup C_2)$; or
\item
$T$ satisfies \textup{(}a\textup{)} and \textup{(}b\textup{)} below for each $m\le n$.
 \begin{enumerate}  
\item $T : \PP_m(C_1) \rightarrow \PP(C_1)\cup\{0\}$
 or $T : \PP_m(C_2) \rightarrow \PP(C_1)\cup\{0\}$;
\item $T : \PP_m(C_1) \rightarrow \PP(C_2)\cup\{0\}$ 
or $T : \PP_m(C_2) \rightarrow \PP(C_2)\cup\{0\}$.
\end{enumerate}
\end{enumerate}
Additionally, $T:\PP_n(\overline{C_1}\cup\overline{C_2})\to\PP(\overline{C_1}\cup \overline{C_2})$ if and only if $T:\PP_n(C_1\cup C_2)\to\PP(C_1\cup C_2)$ and $T[z^k]\in\PP(\overline{C_1}\cup \overline{C_2})$, where $k = \min\{ m :T[z^m]\not\equiv 0\}.$
\end{corollary}

\end{document}